\documentclass[12pt]{amsart}

\usepackage[usenames]{color}
\input xypic

\usepackage{verbatim}
\usepackage{url}
\usepackage{bm}
\usepackage{amsmath}
\usepackage[centertags]{amsmath}
\usepackage{amsfonts}
\usepackage{amssymb}
\usepackage{amsthm}
\usepackage{newlfont}
\usepackage{url}
\usepackage{bm}
\usepackage{chngcntr}

\theoremstyle{plain}
\newtheorem{thm}{Theorem}
\newtheorem{cor}[thm]{Corollary}

\newtheorem{lem}[thm]{Lemma}
\newtheorem{lem*}[thm]{Lemma}
\newtheorem{prop}[thm]{Proposition}

\theoremstyle{definition}
\newtheorem{dfn}{Definition}
\theoremstyle{remark}
\newtheorem{rem}{Remark}
\newtheorem{rem*}{Remark}
\newtheorem{ex}[rem]{Example}

\numberwithin{rem}{section} %{subsection}
\numberwithin{dfn}{section} %{subsection}
\numberwithin{equation}{section} %{subsection}
\numberwithin{thm}{section} %{subsection}

\def\!{\operatorname{!}}

\def\F{\mathbb F}
\def\G{\mathbb G}
\def\1{\bold 1}

\def\deg{\operatorname{deg}}

\def\Hom{\operatorname{Hom}}

\def\Ext{\operatorname{Ext}}

\def\rk{\operatorname{rk}}

\def\deg{\operatorname{deg}}

\def\coker{\operatorname{coker}}

\usepackage{amssymb}

\usepackage{enumitem}

\usepackage{graphicx}

\makeatletter
\@namedef{subjclassname@2010}{%
  \textup{2010} Mathematics Subject Classification}
\makeatother

\usepackage[T1]{fontenc}
\theoremstyle{definition}

\numberwithin{equation}{section}

\newcommand{\Der}{\mathrm{Der} }

\newcommand{\Derin}{\Der_{in}}
\newcommand{\up}[2]{#1^{(#2)}}

\newcommand{\lra}{\longrightarrow}

\newcommand{\podwzorem}[2]{\underbrace{#1}\limits_{#2}}
\newcommand{\nadwzorem}[2]{\overbrace{#1}\limits^{#2}}

\newcommand{\uplra}[1]{\stackrel{#1}{\lra}}

\newcommand\numberthis{\addtocounter{equation}{1}\tag{\theequation}}

\begin{document}
%%%%% To ease editing, for IMPAN journals add:

\baselineskip=17pt

%%%%%%%%%%%%%%%%

\title[On $\mathbf{Ext^1}$ ...]{ On $\mathbf{Ext^1}$ for Drinfeld modules}

\author{Dawid Edmund K{\k e}dzierski, Piotr Kraso{\'n}}

\date{\today}

\address{  Institute of Mathematics, Department of Exact and Natural Sciences, University of Szczecin, ul. Wielkopolska 15, 70-451 Szczecin, Poland 
}
\email{dawid.kedzierski@usz.edu.pl}

\address{Institute of Mathematics, Department of Exact and Natural Sciences, University of Szczecin, ul. Wielkopolska 15, 70-451 Szczecin, Poland }
\email{ piotrkras26@gmail.com}


\thanks{}

\maketitle

\newcommand{\tm}{$\mathbf{t}-$}
\newcommand{\tsm}{$\mathbf{t}^{\sigma}-$}

\begin{abstract}
Let $A={\mathbb F}_q[t]$ be the polynomial ring over a finite field ${\mathbb F}_q$ and let $\phi  $ and $\psi$  be $A-$Drinfeld modules. 
In this paper we consider the  group  ${\mathrm{Ext}}^1(\phi ,\psi )$ with the Baer addition. We show that if $\mathrm{rank}\phi >\mathrm{rank}\psi$ then $\mathrm{Ext^1}(\phi,\psi)$ has the structure of a \tm module. We give complete algorithm describing this structure. We generalize this to the cases: $\mathrm{Ext^1}(\Phi,\psi)$ where $\Phi$ is a \tm module and $\psi$ is a Drinfeld module and $\mathrm{Ext^1}(\Phi, C^{\otimes e})$ where $\Phi$ is a \tm module and $C^{\otimes e}$ is the $e$-th tensor product of Carlitz module. We also establish duality between $\Ext$ groups for \tm modules
and the corresponding adjoint ${\mathbf t}^{\sigma}$-modules. Finally, we  prove the existence of $"\Hom-\Ext"$ six-term exact sequences for \tm modules and dual \tm motives. As the category of \tm modules is only additive (not abelian) this result is nontrivial.
\end{abstract}

\section{Introduction}\label{sec1}
Drinfeld modules since their discovery \cite{d74} gained a lot of attention in arithmetic algebraic geometry because of 
their numerous applications e.g. in class field theory, in Langlands conjectures, theory of automorphic forms or Diophantine geometry (see \cite{g96},  \cite{th04} or \cite{f13}). There is a deep analogy between the theory 
of Drinfeld modules over function fields and  theory of elliptic curves over number fields or more generally between \tm modules and abelian varieties over number fields (see \cite{bp20} ). This analogy justifies various attempts to use the theory of Drinfeld modules in cryptology (see \cite{g03}, \cite{n19}).

There are also some significant differences.
For example the Mordell-Weil groups of \tm 
modules are not finitely generated \cite{p95}.
 The category of Drinfeld modules as well as the category of \tm modules are not semisimple i.e. there are non-trivial extensions.  
\begin{rem}\label{rema1}
Since in this paper we consider $\Ext$ groups in several different categories, we indicate the appropriate category by a subscript $\tau, {\mathbb F}_q[t]$ or $\sigma.$
\end{rem}
Let $\mathrm{Ext}^1_{\tau}(B, A)$ be the Bauer group of extensions of \tm modules i.e. the group of exact sequences
\begin{equation}\label{seq}
0\rightarrow A\rightarrow M\rightarrow B\rightarrow 0
\end{equation}
with the usual addition known from homological algebra (cf. \cite{Mac}).

Let $\phi$ and $\psi$ be Drinfeld modules.
In this paper we study the group $\mathrm{Ext}^1_{\tau}(\phi, \psi).$
It turns out that sometimes we can endow the extension group $\mathrm{Ext}^1_{\tau}(A,B)$, for certain specific \tm modules, with a ${\mathbf t}$-module structure.  In our study we apply the method (used in \cite{pr}) of expressing the elements of $\mathrm{Ext}^1_{\tau}(\phi,\psi)$ by certain classes of  biderivations. This idea was originally  introduced by G. Hochschild in his thesis for the study of extensions ( the first  Hochschild cohomology groups ) of associative algebras.
It is  worth mentioning that the concept of describing an extension space as a cokernel of a line map was used by C.M. Ringel for modules over path algebra of a quiver or more generally for K-species \cite{r76}, 
 \cite{r98}.
The content of the paper is as follows.
In Section \ref{prelim} we recall basic definitions and properties of Drinfeld modules and \tm modules.
In Section \ref{extension} we 
identify  the  $\mathrm{Ext}^1_{\tau}(\phi,\psi)$ for Drinfeld modules satisfying ${\mathrm{rk}}\phi > {\mathrm{rk}}\psi$ with
some ${\mathbb F}_q$-subspace of the ring of skew polynomials $K\{\tau\}.$
In Section \ref{examples} we compute an example which show how to endow $\mathrm{Ext}^1_{\tau}(\phi,\psi)$ 
for the case  ${\mathrm{rk}}\phi > {\mathrm{rk}}\psi$ with the structure of a $\mathbf t$-module. 
Section \ref{extension1} is devoted to the proof of general case i.e. that $\mathrm{Ext}^1_{\tau}(\phi,\psi)$ is in the considered case a $\mathbf t$-module. In the proof of Proposition we describe a recursive step which is sufficient for finding the exact formulas or designing a computer program.
In Section \ref{gen1} we generalize the results from Section \ref{extension1} to certain \tm modules that are products of Drinfeld modules. 
A brief discussion of the remaining cases i.e. ${\mathrm{rk}}\phi \leq {\mathrm{rk}}\psi$ is included in Section \ref{general} where also, for a perfect field $K,$ a duality between \tm modules and \tsm modules is established. The Duality Theorem (Theorem \ref{Duality}) allows one to describe $\mathrm{Ext}^1_{\tau}(\phi,\psi)$ for ${\mathrm{rk}}\phi <  {\mathrm{rk}}\psi$ as an ${\mathbb F}_q(t)$-module. In   Section \ref{Phipsi} we discuss $\mathrm{Ext}^1_{\tau}(\Phi,\psi)$ 
where $\Phi$ is a \tm module, $\psi$ is a Drinfeld module and ${\mathrm{rk}}\Phi > {\mathrm{rk}}\psi.$ 
In Section \ref{Carlitz} we describe $\mathrm{Ext}^1_{\tau}(\Phi,C^{\otimes e})$ for a \tm module $\Phi$ such that 
$\rk\Phi\geq 2$ and the $e$-th tensor power of the Carlitz module $C.$ 
In  Section \ref{push} we prove directly the existence of pullbacks and pushouts in the category of \tm modules. We also prove  existence of the six-term $"\Hom-\Ext"$ exact sequences, where the last map is surjective. As the category of \tm modules is not an abelian category these are nontrivial results ( cf. Example \ref{exm} ).  We conclude this section with some consequences of the six-term exact sequences (Proposition \ref{prop:sequence_exts}  and Theorem \ref{thm:last}).
In Section \ref{dmot} we discuss briefly the category of  Anderson dual \tm motives (over a perfect field) in the context 
of  $\Ext^1$ bifunctor.

\section{Preliminaries}\label{prelim} 
Let $p$ be a rational prime and $A={\mathbb F}_q[t]$  the polynomial ring over the finite field with $q=p^m$ elements and let $K$ be a field of characteristic $p$. 
We  say that $K$ is an \textbf{$A-$field} when we fix an ${\mathbb F}_q-$linear homomorphism ${\iota}: A \rightarrow K$ with 
	$\theta:=\iota(t)$. Denote by $K\{\tau\}$  the ring of twisted polynomials in $\tau$  with coefficients in $K$ i.e. satisfying the additional relation 
	$\tau x=x^q\tau$ for $x\in K$, see \cite{g96}.

In \cite{a}  G. Anderson developed  a generalization of the notion of a Drinfeld module called a \tm module. 
\begin{dfn}\label{tmod}
A $d$-dimensional \tm module over an $A-$field $K$ is an ${\mathbb F}_q$ - algebra homomorphism
\begin{equation}\label{tmodule}
{\Phi} : {\mathbb F}_q[t]\rightarrow {\mathrm{Mat}}_d (K\{\tau\}),
\end{equation}
such that ${\Phi}(t)$, as a polynomial in $\tau$  with coefficients in ${\mathrm{Mat}}_d(K)$  is of the following form
\begin{equation}\label{tmod1}
{\Phi}(t)=({\theta}I_d+N)\tau^0+M_1{\tau}^1+\dots +M_r\tau^r,
\end{equation}
where $I_d $ is the identity matrix and $N$ is a nilpotent matrix. 
\end{dfn} 
In general, a \tm module over $K$ is an algebraic group $E$ defined over $K$ and isomorphic over $K$ to $G^d_a$  together with a choice of ${\mathbb{F}}_q$-linear endomorphism $t: E\rightarrow E$ such that $d(t-\theta)^n{\mathrm{Lie}}(E)=0$ for $n$ sufficiently large. Notice that in the last equality $d( \cdot)$ denotes the differential of an endomorphism of the algebraic group $E.$ The choice of  an isomorphism $E\cong G^d_a$ is equivalent to the choice of $\Phi.$
In order to  indicate this choice of coordinates, we write $E=(G_a^d,\Phi)$. 
\begin{rem}
	Since the map $\Phi$ is a homomorphism of ${\mathbb F}_q$-algebras a \tm module is completely determined by the polynomial ${\Phi}_t$ i.e. by the  image of $t.$  In the sequel, for a ${\mathbf t}$-module $\Phi$, we will use the notations ${\Phi}(t)$ and ${\Phi}_t$ interchangeably. 
\end{rem}	
	The  degree $r$ of ${\Phi}_t$ is called the \textbf{rank} of $\Phi$ and it is denoted as $\rk\Phi$.
\begin{rem}	
	 Notice that this definition of a rank is not the usual one. In fact, usually  the rank of $\Phi$ is defined as the rank of the period lattice of $\Phi$ as a $d{\Phi}(A)$-module (cf. \cite[Section t-modules]{bp20}). Since our paper concerns the algebraic site of the theory of ${\mathbf t}$-modules the adapted by us equivalent definition of  rank seems more convenient.
\end{rem}

\begin{dfn} 
	Let $\Phi$ and $\Psi$ be a two \tm modules of dimension $d$ and $e$, respectively. A morphism $f:\Phi\lra \Psi$ of \tm modules over $K$ is a matrix $f\in\mathrm{Mat}_{d\times e}(K\{\tau\})$
	 such that
	$$
	f\Psi(t) = \Phi(t)f.
	$$
\end{dfn}
In general, if $E=(G_a^d,\Phi)$ and $F=(G_a^e,\Psi)$, then a \tm module morphism  $f : F\rightarrow E$ is a
morphism of commutative algebraic groups $f : G^e_a\rightarrow G^d_a$ over $K$ commuting with the action of $A$ i.e:
$$
	f\Psi(t) = \Phi(t)f.
$$

The set of all morphisms $f:  F\rightarrow E$ will be denoted as $\Hom_{\tau}(F,E).$ We decided to add the subscript $\tau$ 
for the consistency with the notation used for the group of extensions i.e. $\Ext^1_{\tau}(F,E).$

Every \tm module $E=(G_a^d,\Phi)$ induces an ${\mathbb F}_q[t]-$module structure on $K^d$, where  multiplication by $t$ is given by  evaluation of $\Phi_t$, i.e. 
$$t*x=\Phi_t(x)\quad \textnormal{for} \quad x\in K^d.$$ 
This $\F_q[t]-$module is called the  Mordell-Weil group of $\Phi$ and it is denoted as $\Phi(K^d).$ Similarly, each morphism $f:E=(G_a^d,\Phi)\lra F=(G_a^e,\Psi)$ of \tm modules induces a morphism $\Phi(f):\Phi(K^d)\lra \Psi(K^e)$ of $\F_q[t]-$modules. Then $\Phi(-)$ is a covariant inclusion functor from the category of \tm modules  to the category of $\F_q[t]-$modules. In the Example  \ref{exm} we will see that this functor is not full. 

 Dimension one \tm modules are called \textbf{Drinfeld} modules and the Drinfeld module $C:\F_q[t]\lra K\{\tau\}$, given by the formula $C(t)=\theta+\tau$ is called the \textbf{Carlitz module}.

We also consider the \textbf{zero \tm module} of the form $0:\F_q[t]\lra 0.$ Then the category of \tm modules becomes an $\F_q[t]-$linear additive category and  the notion of a short exact sequence (\ref{seq}) makes sense.

Recall from \cite{pr}, that each extension of a $\mathbf t$-module $\Phi:\F_q[t]\lra {\mathrm{Mat}}_d(K\{\tau\})$ by $\Psi:\F_q[t]\lra {\mathrm{Mat}}_e(K\{\tau\})$ can be determined by an $\F_q-$linear map $\delta:\F_q[t]\lra {\mathrm{Mat}}_{e\times d}(K\{\tau\})$ such that
\begin{equation}\label{delta}	
	\delta(ab)=\Psi(a)\delta(b)+\delta(a)\Phi{(b)}\quad \textnormal{for all}\quad a,b\in\F_q[t].
\end{equation}
	
	Such maps are called biderivations, and we will denote the $\F_q-$vector space of all biderivations by $\Der(\Phi, \Psi)$. 
	In the sequel we will, as usual,  denote ${\Phi}_a\,:={\Phi(a)}$, ${\Psi}_a\,:={\Psi(a)}$ etc..
	It is easy to check that the biderivation $\delta$ is uniquely determined, by the value $\delta(t)\in  {\mathrm{Mat}}_{e\times d}(K\{\tau\})$. Then the map $\delta\mapsto \delta(t)$ induces the $\F_q-$linear isomorphism of the vector spaces $\Der(\Phi, \Psi)$ and $ {\mathrm{Mat}}_{e\times d}(K\{\tau\})$. Let $\delta^{(-)}: {\mathrm{Mat}}_{e\times d}(K\{\tau\})\lra \Der(\Phi, \Psi)$ be an $\F_q-$linear map defined by the following formula:
	$$\delta^{(U)}(a)=U\Phi_a - \Psi_aU\quad \textnormal{for all}\quad a\in \F_q[t]\quad\textnormal{and}\quad U\in {\mathrm{Mat}}_{e\times d}(K\{\tau\}).$$
	The image of the map  $\delta^{(-)}$ is denoted by $\Derin(\Phi, \Psi)$, and the elements of $\Derin(\Phi, \Psi)$ are called  inner biderivations. In addition, both $\Ext_{\tau}^1(\Phi,\Psi)$ and $\Der(\Phi, \Psi)/\Derin(\Phi, \Psi)$ have  $\F_q[t]$-module structures. Then from \cite[Lemma 2.1]{pr} there is an $\F_q[t]-$module isomorphism 
	\begin{align}\label{iso_ext}
		\Ext^1_{\tau}(\Phi,\Psi)\cong\mathrm{coker}\delta^{(-)}=\Der(\Phi, \Psi)/\Derin(\Phi, \Psi).
	\end{align}

	Because the map $\delta\mapsto \delta(t)$ is an isomorphism, we will identify the coset $\delta+\Derin(\Phi, \Psi)$ with the coset $\delta(t)+\Derin(\Phi_t, \Psi_t)$ in ${\mathrm{Mat}}_{e\times d}(K\{\tau\})$, where
	$\Derin(\Phi_t, \Psi_t)=\Big\{ \delta(t)\mid \delta\in \Derin(\Phi, \Psi)\Big\}$.

\begin{rem}\label{notation} 
In what follows we adopt the following notation:	$\up ci:=c^{q^i}$ is the evaluation of the Frobenius twist $\tau^i$ on an element $c\in K$. In particular $\up c0=c.$ 
\end{rem}

	\section{Extension of Drinfelds modules with t-modules structures}\label{extension}

	In this section we  study extensions of Drinfeld modules using  biderivations.  Our aim is to equip the extension space $\Ext^1_{\tau}(\phi, \psi)$ with the natural $\mathbf t$-module structure. 
	This is generalization of M.A Papanikolas and N. Ramachandran results concerning the case where $\psi$ is the Carlitz module. Further, we will give some basic properties of these $\mathbf t$-modules.

	\begin{lem}\label{lemma_iso_Ext}
		Let $\phi$ and $\psi$ be Drinfeld modules, such that $\rk\phi>\rk\psi$. Then there is an $\F_q-$linear isomorphism  
		$$\Ext^1_{\tau}(\phi, \psi)\cong K\{\tau\}_{<\rk\phi}:= \Big\{w\in K\{\tau\}\mid \deg_\tau w<\rk\phi   \Big\}$$ 
	\end{lem}
	\begin{proof}
	We will identify the biderivation $\delta$ with the polynomial $w(\tau)=\delta(t)\in K\{\tau\}$. 
	Let $\phi_t=\theta+\sum\limits_{i=1}^na_i\tau^i$ and $\psi_t=\theta+\sum\limits_{j=1}^mb_j\tau^j$, then $n>m$. 	From isomorphism \eqref{iso_ext} it suffices to show that 
	$\Der(\phi, \psi)/\Derin(\phi, \psi)\cong K\{\tau\}_{<\rk\phi}$. 
 To prove this, we  consider generators of the space $\Derin(\phi, \psi)$ of the form
 $\delta^{(c\tau^k)}$ where  $k=0,1,2,\dots$ and $c\in K.$ We have:
	\begin{align}\label{iso1}
	\begin{split}
		\delta^{(c\tau^k)}(t)=c\tau^k\phi_t-\psi_t c\tau^k= c\tau^k\Big(\theta+\sum\limits_{i=1}^na_i\tau^i \Big)- \Big( \theta+\sum\limits_{j=1}^mb_j\tau^j \Big)c\tau^k\\
		=c \up{\theta}k \tau^k+\sum\limits_{i=1}^nc\up{a_i}k\tau^{i+k}-c\theta \tau^k-\sum\limits_{j=1}^m\up cj b_j\tau^{j+k}\\
		= c(\up\theta k-\theta)\tau^k+\sum\limits_{j=1}^m\Big(c\up{a_j}k-\up cj b_j\Big)\tau^{j+k}+ \sum\limits_{j=m+1}^{n}  c\up{a_i}k\tau^{i+k}\\
		= c(\up\theta k-\theta)\tau^k+\sum\limits_{j=k+1}^{m+k}\Big(c\up{a_{j-k}}k-\up c{j-k} b_{j-k}\Big)\tau^{j}+ \sum\limits_{j=m+k+1}^{n+k}  c\up{a_{j-k}}k\tau^{j}.
		\end{split}
	\end{align}
	Because $\deg_\tau \delta^{(c\tau^k)}(t)=n+k$  for $c\neq 0$,  each non-zero inner biderivation is given by a polynomial with $\tau-$degree greater or equal to $n$. Therefore, if $w,\widehat{w}\in K\{\tau\},\,\,w\neq \widehat{w}$ are such that $\deg_\tau w<n$ and $\deg_\tau \widehat{w}< n$, then $w$ and $\widehat{w}$ represent different cosets in $\Der(\phi, \psi)/\Derin(\phi, \psi)$. On the other hand given a polynomial $w(\tau) =\sum\limits_{i=1}^{n+l} a_i\tau^i\in K\{\tau\}$ with $l>0$, one can find $c_l$ (cf. (\ref{iso1}))  such that the polynomial  $w(\tau)-\delta^{(c_l\tau^l)}$ has $\tau$-degree at most 
	$n+l-1$ and represents the same class as $w$ in  $\Der(\phi, \psi)/\Derin(\phi, \psi)$. By downward induction we obtain:
	$$\widetilde w(\tau)=w(\tau)-\delta^{(c_l\tau^l)}-\cdots-\delta^{(c_1\tau^1)}-\delta^{(c_0)}$$	 where $\deg_\tau\widetilde w<n$ and both $w$ and $\widetilde w$ represent the same coset. 
	
	Therefore the isomorphism $\delta\mapsto \delta(t)$ induces the  isomorphism of $\F_q-$vector spaces $\Der(\phi, \psi)/\Derin(\phi, \psi)\cong K\{\tau\}_{<\rk\phi}$.
	\end{proof}
\begin{cor}
Let $\delta\in  K\{\tau\}_{<\rk\phi}$ then $\delta$ corresponds to the extension:
\begin{equation}\label{expl}
0\rightarrow \psi \rightarrow \Gamma \rightarrow \phi\rightarrow 0
\end{equation}
where $\Gamma$ is given by the matrix $\Gamma_t=\begin{bmatrix}
\phi_t& 0  \\
\delta & \psi_t
\end{bmatrix}.$  Morever, this correspondence is  compatible with the Baer sum  of extensions i.e. 
if $\delta_1,\,\, \delta_2 \in  K\{\tau\}_{<\rk\phi}$ correspond to $\Gamma^1_t=\begin{bmatrix}
\phi_t& 0  \\
\delta_1 & \psi_t
\end{bmatrix}$   and $\Gamma^2_t=\begin{bmatrix}
\phi_t& 0  \\
\delta_2 & \psi_t
\end{bmatrix}
$  respectively then the Baer sum of extensions $\Gamma^1$ and $\Gamma^2$ corresponds to 
$\Gamma$  such that $\Gamma_t=\begin{bmatrix}
\phi_t& 0  \\
\delta_1+ \delta_2& \psi_t
\end{bmatrix}.
$
\qed
\end{cor}

\begin{rem}
Notice that $\Ext^1_{\tau}(\phi,\psi)$ is a reflexive bilinear space with the bilinear product $<\delta_1,\delta_2>=\sum_{i=0}^{n-1}a_ib_i$ where $\delta_1=a_0+\dots+a_{n-1}\tau^{n-1}$ and $\delta_2=b_0+\dots+b_{n-1}\tau^{n-1}.$
\end{rem}

	Let $\phi$ and $\psi$ be Drinfeld modules and let $r:=\rk\phi-\rk\psi>0$.
	Recall from \cite{pr} that the structure of the $\F_q[t]$-module on $\Der(\phi, \psi)/\Derin(\phi, \psi)$ is defined by the following formula:
	\begin{equation}\label{modstr} 
	a*\Big(\delta+\Derin(\phi, \psi)\Big):=\psi_a \delta+\Derin(\phi, \psi)
	\end{equation}
	for $a\in\F_q[t]$ and $\delta\in\Der(\phi, \psi)$. 
	\begin{rem}\label{simplification}
	For  simplicity, we  omit the notation $+\Derin(\phi, \psi)$ when considering the coset $\delta+\Derin(\phi, \psi)$ in the space $\Der(\phi, \psi)/\Derin(\phi, \psi)$.
	\end{rem}

	\section{An easy but essential example}\label{examples}
	
	In this section we give an example to illustrate how to determine the structure of a $\mathbf t$-module on the space of extensions of Drinfeld modules.  
	\begin{ex}\label{ex_first}
		Let ${\phi}_t=\theta +{\tau}^3$ and $\psi_t=\theta+\tau^2.$ By Lemma \ref{lemma_iso_Ext}  $$\Ext^1_{\tau}(\phi, \psi)\cong K\{\tau\}_{<3}= \{ c_0+c_1\tau +c_2\tau^2 \mid c_i\in K \}.$$ 
		In order to equip $\Ext^1_{\tau}(\theta +{\tau}^3,\theta+\tau^2)$ with the $\mathbf t$-module structure we will transfer the structure of the $\F_q[t]-$module from  $\Ext^1_{\tau}(\theta +{\tau}^3,\theta+\tau^2)$ to the space $K\{\tau\}_{<3}$ via the isomorphism from Lemma \ref{lemma_iso_Ext}. Since each element of $K\{\tau\}_{<3}$ is of the form $\sum\limits_{i=0}^{2}c_i\tau^i$, it is sufficient to determine the value of multiplication by $t$ on the generators $c_i\tau^i$ for $i=0,1,2$, where $c_i\in K$. 
		
		From \eqref{modstr} the value $t*(c_0)$ can be computed in the following way:  
		$$t*(c_0)=\psi_t\cdot c_0=\big(\theta+\tau^2\big)\cdot c_0=\theta c_0+\up{c_0}2\tau^2$$
		Similarly,
		$$t*(c_1\tau)=\psi_t\cdot c_1\tau=\big(\theta+\tau^2\big)\cdot c_1\tau=\theta c_1\tau+\up{c_1}2\tau^3.$$
		Because the polynomial $\theta c_1\tau+\up{c_1}2\tau^3$ does not belong to $K\{\tau\}_{<3}$,  like in the proof of the Lemma \ref{lemma_iso_Ext}, we  reduce the term $\up{c_1}2\tau^3$ by the generator 
		\begin{align*}
			\delta^{(c)}(t)&=c\phi_t-\psi_t c= c(\theta +{\tau}^3)-(\theta +{\tau}^2)c=-\up c2\tau^2+c\tau^3
		\end{align*}
		for $c=\up{c_1}2$. Then 
		 $$t*(c_1\tau)=\theta c_1\tau+\up{c_1}2\tau^3-\delta^{(c)}=\theta c_1\tau+\up {c_1}4\tau^2.$$
		Next,
		$$t*(c_2\tau^2)=\psi_t\cdot c_2\tau^2=\big(\theta+\tau^2\big)\cdot c_2\tau^2=\theta c_2\tau^2+\up{c_2}2\tau^4.$$
		The term $\up{c_1}2\tau^4$ can be reduced by the generator
		\begin{align*}
			\delta^{(c\tau)}(t)&=c\tau\phi_t-\psi_t c\tau= c\tau(\theta +{\tau}^3)-(\theta +{\tau}^2)c\tau\\
			&=
			-c\big(\up\theta1-\theta\big)\tau-\up c2\tau^3+c\tau^4,
		\end{align*} for $c=\up{c_2}2$. Therefore 
	$$t*(c_2\tau^2)=t*(c_2\tau^2)-\delta^{(c\tau)}=\up{c_2}2\big(\theta-\up\theta1\big)\tau+\theta c_2\tau^2+ \up{c_2}4\tau^3.$$
	The term $\up{c_2}4\tau^3$ can be reduced by the generator $\delta^{(c)}$ for $c=\up{c_2}4$.
	Hence 
	\begin{align*}
		t*(c_2\tau^2)&=t*(c_2\tau^2)-\delta^{\big(\up{c_2}2 \tau\big)}-\delta^{\big(\up{c_2}4\big)}\\
		&=\up{c_2}2\big(\theta-\up\theta1\big)\tau+\Big(\theta c_2+\up{c_2}6\Big)\tau^2.
	\end{align*}
Now, choose the basis $e_0=1, e_1=\tau, e_2=\tau^2$ in $K\{\tau\}_{<3}$. In this basis the value $t*(c_0)$ has the following coordinates 
	\begin{align*}
		t*(c_0)&=\theta c_0+\up{c_0}2\tau^2= \big[\theta c_0, 0,  \up{c_0}2\big].
	\end{align*}
Notice that $\up{c_0}2$ is the value of the polynomial $\tau^2$ at $c_0$. Let $\tau^2\mid_{c_0}:=\up{c_0}2$. Then we can express $t*(c_0)$ in the following form 
$$	t*(c_0)= \big[\theta, 0,  \tau^2 \big]_{\mid c_0}.$$
In a similar way we obtain: 
		\begin{align*}
			t*c_0=c_0\theta+c_0^{(2)}\tau^2=[\theta, 0, \tau^2]_{\mid c_0}\\
			t*c_1\tau=c_1\theta\tau+c_1^{(4)}\tau^2=[0, \theta,  \tau^4]_{\mid c_1}\\
			t*c_2\tau^2=c_2^{(2)}(\theta - \theta^{(1)})\tau+(c_2\theta+c_2^{(6)})\tau^2=[0, (\theta -\theta^{(1)})\tau^2,  \theta+\tau^6]_{\mid c_2}
		\end{align*}
Then the multiplication by $t$ on $K\{\tau\}$ can be expressed by the following matrix:
		\begin{equation}\label{pi1}
			\Pi_t=\begin{bmatrix}
				\theta& 0 & 0\\
				0& \theta & \theta-\theta^{(1)}\\
				\tau^2 & \tau^4 & \theta+\tau^6
			\end{bmatrix}
		\end{equation}
	This matrix induces the homomorphism of $\F_q-$algebras
	$$\Pi:\F_q[t]\lra {\mathrm{Mat}}_3(K\{\tau\}),$$
	such that 
	$$\Pi_t=\theta I_3+
\begin{bmatrix}
	0& 0 & 0\\
	0&0 & \big(\theta-\up\theta 1\big)\\
	1& 0 & 0
\end{bmatrix}\tau^2
+\begin{bmatrix}
	0& 0 & 0\\
	0&0 & 0\\
	0& 1 & 0
\end{bmatrix}\tau^4+
\begin{bmatrix}
	0& 0 & 0\\
	0&0 & 0\\
	0& 0 & 1
\end{bmatrix}\tau^6.$$
Therefore $\Pi$ gives  rise to  a $\mathbf t$-module structure on $\Ext^1_{\tau}(\theta +{\tau}^3,\theta+\tau^2)$.
\end{ex}

\section{$\Ext^1_{\tau}(\phi, \psi)$ for Drinfeld modules $\phi$ and $\psi$ with $\rk\phi>\rk \psi$  }\label{extension1}
We have the following:
\begin{prop}\label{prop:ext_as_t_module}
	Let $\phi$ and $\psi$ be Drinfeld modules such that $\rk\phi>\rk \psi$. Then $\Ext^1_{\tau}(\phi, \psi)$ has a natural structure of a $\mathbf t$-module. This is given by the  map
	$\Pi:\F_q[t]\lra {\mathrm{Mat}}_{\rk\phi}(K\{\tau\})$, where
	\begin{equation} \label{macierz_dla_exta}
		\Pi_t=\left[\begin{array}{c|ccc}
			\theta & 0&\dots & 0\\\hline
			\delta_1 & &&\\
			\vdots &  &\Pi^0_t &\\
			\delta_{n-1} & &&
		\end{array}\right].
	\end{equation}
\end{prop}
At this moment, we will focus on the proof that the $\mathbf t$-module structure comes from the matrix \eqref{macierz_dla_exta}. Next we will see that the matrix $\Pi^0_t$ gives  rise to a $\mathbf t$-module structure on $\Ext^1_{0,{\tau}}(\phi, \psi)$ and the vector $[\delta_1,\cdots,\delta_{n-1}]^t$ determines the extension of  $\Ext^1_{0,\tau}(\phi, \psi)$  with $\Ext^1_{\tau}(\phi, \psi)$ as the middle term. 
\begin{proof}
	We will describe an algorithm that  allows us to determine the structure of the $\mathbf t-$module on $\Ext^1_{\tau}(\phi, \psi)$. 	
	From  the Lemma \ref{lemma_iso_Ext} we know that 
	$$\Ext^1_{\tau}(\phi, \psi)\cong \Big\{c_0+c_1\tau +\cdots + c_{{\rk\phi}-1}\tau^{{\rk\phi}-1}\mid c_i\in K \Big\}$$ as an $\F_q-$vector space.
	In order to equip $\Ext^1_{\tau}(\phi, \psi)$ with the $\mathbf t$-module structure we will transfer the structure of the $\F_q[t]-$module form  $\Ext^1_{\tau}(\phi, \psi)=\Der(\phi, \psi) /\Derin(\phi, \psi)$ to the space $K\{\tau\}_{<\rk\phi}$ via the isomorphism from Lemma \ref{lemma_iso_Ext}. Since each element of $K\{\tau\}_{<{\mathrm{rk}}\phi}$ is of the form $\sum\limits_{i=0}^{\rk\phi-1}c_i\tau^i$, it is sufficient to determine the value of multiplication by $t$ on the generators $c_i\tau^i$ for $i=0,1,\dots, \rk\phi-1$, where $c_i\in K$. 
	
	Next we will choose the coordinate system $e_i=\tau^i$ for $i=0,1,2,\dots, \rk\phi-1$ in $K\{\tau\}_{<\rk\phi}$ and see that the matrix $\Pi_t$ of the multiplication map $t*(-):K\{\tau\}_{<\rk\phi}\lra K\{\tau\}_{<\rk\phi}$ gives  rise to the $\mathbf t$-module structure on $\Ext^1_{\tau}(\phi, \psi).$ This structure is given by  \eqref{macierz_dla_exta}. 
	
	Assume that $\phi_t=\theta+\sum\limits_{i=1}^na_i\tau^i$, $\psi_t=\theta+\sum\limits_{j=1}^mb_j\tau^j$ and $r:=n-m>0$.
	Then from the formula \eqref{modstr} we see that 
	$$t*(c_i\tau^i)=\psi_t\cdot c_i\tau^i\quad \textnormal{for}\quad c_i\in K\quad \textnormal{and}\quad i=0,1,2,\dots, \rk\phi-1.$$   
	Suppose that $i\in\{0,1,\dots, r-1\}.$ Then 
		\begin{align*}
		t*(c_{i}\tau^{i})&=\Big(\theta+\sum\limits_{j=1}^mb_j\tau^j\Big) c_{i}\tau^{i}
		= c_{i}\theta\tau^{i}+\sum\limits_{j=1}^m b_j\up{c_{i}}j\tau^{j+i}\in K\{\tau\}_{<\rk\phi}
	\end{align*}	
	Hence 
	\begin{align*}
		t*(c_{i}\tau^{i})&=\Big[\nadwzorem{0,\dots,0}{i-1}, \theta, b_1\up{c_{i}}1, b_2\up{c_{i}}2,\dots, b_m\up{c_{i}}m,0,\dots 0 \Big]=\\
		&=\Big[\podwzorem{0,\dots,0}{i-1}, \theta, b_1\tau, b_2\tau^2,\dots, b_m\tau^m,0,\dots 0 \Big]_{|_{c_i}}=
	\end{align*}
Therefore the first $r$ column of the matrix $\Pi_t$ satisfy our claim. 

 Consider the case   $i\in\{r,r+1,\dots, \rk\phi-1\}$. In this situation the element
$\psi_t\cdot c_i\tau^i$ has $\tau-$degree bigger than $\rk\phi-1$, so like in the proof of the Lemma \ref{lemma_iso_Ext} we    reduce the monomial with the biggest $\tau-$degree by the generator $\delta^{(c\tau^k)}\in\Derin(\phi,\psi)$ and  continue this procedure until we get the reduced polynomial belonging to $K\{\tau\}_{<\rk\phi}$.

Then we perform downward induction from $n+i-r$ to $n$ and see that at each step after the reduction we obtain polynomials satisfying the following properties:
\begin{itemize}
	\item[(i)] the  term at $\tau^0$ of the reduced polynomial is equal to zero,
	\item[(ii)] each coefficient at $\tau^l$ of the reduced polynomial can be written as the evaluation at $c_i$ of a skew polynomial $w_l(\tau),$
	\item[(iii)] if $l\neq i$, then the skew polynomials $w_l(\tau)$ from $(ii)$ have no free term,
	\item[(iv)] the skew polynomial $w_i(\tau)$ has a free term equal to $\theta.$
\end{itemize}      
From $(i)$ we see that the matrix $\Pi_t$ is of the form as claimed. On the other hand  $(ii)$ implies that after the reduction procedure is completed we can present the coefficients in the chosen coordinate system $e_i=\tau^i$, as the values of the skew polynomials in $c_i$.  Write the matrix $\Pi_t$ in the following form
$$\Pi_t=(\theta \cdot I+N)\tau^0+\sum\limits_{i=1}^{\textnormal{finite} }A_i\tau^i,\quad\textnormal{where} \quad I,N,A_i\in {\mathrm{Mat}}_{\rk\phi}(K).$$
The condition  $(iii)$ implies that the matrix $N=0$ and the condition $(iv)$ implies that $I$ is identity matrix. Therefore $\Pi_t$ yields the $\mathbf t$-module structure on $\Ext^1_{\tau}(\phi,\psi)$ \\
\textbf{Induction:} At the start of the induction we consider the following polynomial:
	\begin{equation}\label{eq:multiplication_t__on_generator}
			t*(c_{i}\tau^{i})
		= c_{i}\theta\tau^{i}+\sum\limits_{j=1}^m b_j\up{c_{i}}j\tau^{j+i},
	\end{equation}
that satisfies conditions $(i)-(iv)$ in an obvious way. 

Suppose we have made $k$ reductions that satisfy conditions $(i)-(iv)$. Therefore the value $t*(c_i\tau^i)$ after reductions can be written in the following form:
\begin{align*}
	w_{i-r-k+1}(c_i)\tau^{i-r-k+1}+ w_{i-r-k+2}(c_i)\tau^{i-r-k+2}+\cdots+w_{m+i-k}(c_i)\tau^{m+i-k},
\end{align*}  
where the polynomials $w_l(\tau)$ satisfy conditions $(i)-(iv)$.
We  reduce the term  $w_{m+i-k}(c_i)\tau^{m+i-k}$ by the  generator
$\delta^{(c\tau^{i-r-k})}$ where $c=\dfrac{w_{m+i-k}(c_i)}{\up{a_n}{i-r-k}}$. 
Because all non-zero terms of $\delta^{(c\tau^{i-r-k})}$ have $\tau-$degree $\geq 1$, then $(i)$ is obvious. Recall that from
(\ref{iso1}) the generator $\delta^{(c\tau^{i-r-k})}$ can be rewritten in the following form:
\begin{align*}
		\delta^{(c\tau^{i-r-k})}(t)&
	= \sum\limits_{j=i-r-k}^{m+i-k-1}p_j(c_i)\tau^{j}+  \podwzorem{c\up{a_n}{i-r-k}}{=w_{m+i-k}(c_i)}\tau^{m+i-k},
\end{align*}
where the polynomials $p_l(\tau)\in K\{\tau\}$ satisfy the conditions $(i)-(iii)$.
Therefore after the reduction by the generator $	\delta^{(c\tau^{i-r-k})}$
we obtain the following form of  $t*(c_i\tau^i):$
\begin{align*}
	-p_{i-r-k}(c_i)\tau^{i-r-k}+ \sum\limits_{j=i-r-k+1}^{m+i-k-1}\big(w_j(c_i) -p_j(c_i)\big)\tau^{j}.
\end{align*}
We put $\widehat{w_{i-r-k}}(\tau)=-p_{i-r-k}(\tau)$, $\widehat{w_{j}}(\tau)=w_j(\tau) -p_j(\tau)$ for $j=i-r-k+1,\dots, m+i-k-1$. Because the polynomials $p_l(\tau)$ and $w_l(\tau)$ satisfy the conditions $(i)-(iii)$ and $w_i(\tau)$ satisfy $(iv)$, then the polynomials $\widehat{w_{l}}(\tau)$ also satisfy these four conditions. This completes the induction step, and thus we proved that $\Ext^1_{\tau}(\phi, \psi)$ has a $\mathbf t-$module structure. 
\end{proof}

For $A\in {\mathrm{Mat}}_{n_1\times n_2}(K\{\tau\})$ let $dA\in {\mathrm{Mat}}_{n_1\times n_2}(K)$    be the constant term of  
$A$ viewed as a polynomial in $\tau.$ 
For $\mathbf t-$ modules $\Phi$ and $\Psi$ let
\begin{equation}\label{der0}
\Der_0(\Phi,\Psi)=\{ \delta \in \Der(\Phi,\Psi)\,\,\mid \,\, d\delta (t)=0\}.
\end{equation}
Following \cite{pr} define $\Ext^1_{0,\tau}(\Phi, \Psi):=\Der_0(\Phi,\Psi)/ \Der_0(\Phi,\Psi)\cap \Der_{in}(\Phi,\Psi).$ 
We have the following:
\begin{lem}\label{lem:ext0}
		Let $\phi$ and $\psi$ be the  Drinfeld modules, such that $r=\rk\phi-\rk\psi>0$. Then 
		\begin{itemize}
			\item[$(i)$] there exists an isomorphism of  $\F_q[t]-$modules between $\Ext^1_{0,\tau}(\phi, \psi)$   and $K\{\tau\}_{\langle 1, \rk\phi )}=\Big\{ \sum\limits_{i=1}^{\rk\phi-1} c_{i}\tau^{i}\mid c_i\in K \Big\},$
			\item[$(ii)$] $K\{\tau\}_{\langle 1, \rk\phi )}$ is an $\F_q[t]-$submodule of $K\{\tau\}_{<\rk\phi},$
			\item[$(iii)$] $\Ext^1_{0,\tau}(\phi, \psi)$ is an $\F_q[t]-$submodule of $\Ext^1_{\tau}(\phi, \psi),$
			\item[$(iv)$] $\Ext^1_{0,\tau}(\phi, \psi)$ has a natural structure of a $\mathbf t$-module.		
		\end{itemize}		 
\end{lem}
\begin{proof}
	The proof of part $(i)$ is similar to the proof of Lemma \ref{lemma_iso_Ext}. Recall, that  $\Ext^1_{0,\tau}(\phi, \psi)=\mathrm{Der}_0(\phi, \psi)/ \Derin(\phi, \psi)\cap \mathrm{Der}_0(\phi, \psi)$ and $ \delta^{(c\tau^k)}\in\mathrm{Der}_0(\phi, \psi)$ for all $k\geq 0$. Because $\deg_{\tau}(\delta^{c\tau^k}_t)=\rk\phi+k$  each biderivation $\delta\in \mathrm{Der}_0(\phi, \psi)$ can be reduced as in the proof of the Lemma  \ref{lemma_iso_Ext}. This implies \nolinebreak$(i)$. \\
 The vector space $K\{\tau\}_{\langle 1, \rk\phi )}$ has the $\F_q[t]-$module structure coming from the same formulas $t*\big(c_k\tau^k\big)$ for $k=1,2,\dots, n-1$ as in the case  $K\{\tau\}_{<\rk\phi}$. Therefore the standard inclusion $K\{\tau\}_{\langle 1, \rk\phi )}\hookrightarrow K\{\tau\}_{<\rk\phi}$ is the inclusion of $\F_q[t]-$modules. 
 This establishes $(ii).$\\
 $(iii)$ follows from $(i)$ and $(ii)$.\\
	Part $(iv)$ follows from the proof of the Proposition \ref{prop:ext_as_t_module}, where the matrix $\Pi^0_t$ from \eqref{macierz_dla_exta} induces the homomorphism  
	$\Pi^0:\F_q[t]\lra {\mathrm{Mat}}_{n-1}(K\{\tau \})$ giving rise to the $\mathbf t$-module structure on $\Ext^1_{0,\tau}(\phi, \psi)$.  
\end{proof}

The following theorem is a specialization of \cite[Lemma 2.2]{pr}.  However, we are able to describe the maps in an explicit way.
\begin{thm}\label{thm:short_sequence_t_modules}
		Let $\phi$ and $\psi$ be  Drinfeld modules, such that $r=\rk\phi-\rk\psi>0$. Then there is a short exact sequence of \tm modules
		\begin{equation}\label{ess}
		0\lra \Ext^1_{0,\tau}(\phi, \psi)\lra \Ext^1_{\tau}(\phi, \psi) \lra \G_a\lra 0.
		\end{equation}
\end{thm}
\begin{proof}
	Let $\Pi^0_t$ and $\Pi_t$ be the matrices giving the $\mathbf t$-module structures for $\Ext^1_{0,\tau}(\phi, \psi)$ and $\Ext^1_{\tau}(\phi, \psi)$, respectively. Then the matrix 
	$$\left[\begin{array}{c}
		0\dots 0\\ \hline
		\\
		I_{\rk\phi-1}\\	\\	
	\end{array}\right]$$   induces the inclusion of $\mathbf t$-modules $i:\Ext^1_{0,\tau}(\phi, \psi)\hookrightarrow\Ext^1_{\tau}(\phi, \psi)$. It is easy to check that $\coker(i)$ is the trivial $\mathbf t$-module $\G_a$. 
\end{proof}

\begin{rem}
Notice that the biderivation corresponding to the extension (\ref{ess}) comes from the first column of  $\Pi_t.$
\end{rem}

\section{$\Ext^1_{\tau}\Big(\prod_{i=1}^n\phi_i, \prod_{j=1}^m\psi_j\Big)$}\label{gen1}
We have the following generalization of  Proposition \ref{prop:ext_as_t_module}:
.
\begin{thm}\label{thm:product_ext}
		Assume that $\phi_i$, $\psi_j$ for $i=1,\dots, n$, $j=1,\dots, m$ are Drinfeld modules, such that  $\rk \phi_i>\rk\psi_j$ for all $i,j$. Then 
		$\Ext^1_{\tau}\Big(\prod_{i=1}^n\phi_i, \prod_{j=1}^m\psi_j\Big)$ has a natural structure of a $\mathbf t$-module.
	\end{thm}
	\begin{proof}
		Let $\phi_i(t)=\theta+\sum\limits_{k=1}^{n_i}a_{i,k}\tau^k$ for $i=1,\dots, n$ and
		$\psi_j(t)=\theta+\sum\limits_{k=1}^{m_j}b_{j,k}\tau^k$ for $j=1,\dots, m$. We put 
		$$N=\max\{n_i\mid i=1,\dots, n\}\quad \textnormal{and}\quad M=\max\{m_j\mid j=1,\dots, m\}.$$
		 Denote by $E_{j\times i}$ the $m\times n$  the elementary matrix with the only one non-zero element equal to $1$ at $j\times i$ place. The elements $\delta^{(c\tau^rE_{j\times i})}$ generate $\Derin\Big(\prod_{i=1}^n\phi_i, \prod_{j=1}^m\psi_j\Big)$, where
		\begin{align*}
			\delta^{(c\tau^rE_{j\times i})}(t)&=E_{j\times i}\cdot\podwzorem{\Bigg(c\Big(\up{\theta}r-\theta\Big)\tau^r+\sum_{l=1}^Nc\up{a_{i,l}}r\tau^{r+l}-\sum_{l=1}^M \up cl b_{j,l}\tau^{l+r}\Bigg)}{\in \Derin(\phi_i, \psi_j)}
		\end{align*} 
	Then $$\Derin\Big(\prod_{i=1}^n\phi_i, \prod_{j=1}^m\psi_j\Big)=
	\left[\begin{array}{ccc}
		\Derin(\phi_1,\psi_1)&\cdots&\Derin(\phi_n,\psi_1)\\
		\Derin(\phi_1,\psi_2)&\cdots&\Derin(\phi_n,\psi_2)\\
		\vdots& \ddots&\vdots\\
		\Derin(\phi_1,\psi_m)&\cdots&\Derin(\phi_n,\psi_m)\\
	\end{array}\right],$$
	and therefore
	\begin{equation}\label{eq:extension_isomorphism_product}
		\Ext^1_{\tau}\Big(\prod_{i=1}^n\phi_i, \prod_{j=1}^m\psi_j\Big)\cong
		\left[\begin{array}{ccc}
			\Ext^1_{\tau}(\phi_1,\psi_1)&\cdots&\Ext^1_{\tau}(\phi_n,\psi_1)\\
			\Ext^1_{\tau}(\phi_1,\psi_2)&\cdots&\Ext^1_{\tau}(\phi_n,\psi_2)\\
			\vdots& \ddots&\vdots\\
			\Ext^1_{\tau}(\phi_1,\psi_m)&\cdots&\Ext^1_{\tau}(\phi_n,\psi_m)\\
		\end{array}\right],
	\end{equation}
	as $\F_q-$linear spaces. \\
	Moreover, for $a\in\F_q[t]$ and a biderivation $\delta\in \Der\Big(\prod_{i=1}^n\phi_i, \prod_{j=1}^m\psi_j\Big)$ we have
	\begin{align*}
		a*\delta&=\Bigg(\prod_{j=1}^m\psi_j\Bigg) (a)\cdot \delta =\\
		&= \left[\begin{array}{cccc}
			\psi_1(a)&0&\cdots&0 \\
			0&\psi_2(a)&\cdots& 0\\
			\vdots& \vdots& \ddots& \vdots\\
			0& 0& \cdots & \psi_{m}(a)
		\end{array}\right]\cdot
	\left[\begin{array}{ccc}
		\delta_{1\times 1}(\tau)&\cdots& \delta_{1\times n}(\tau) \\
		\delta_{2\times 1}(\tau)&\cdots& \delta_{2\times n}(\tau)\\
		\vdots&  \ddots& \vdots\\
		\delta_{m\times 1}(\tau) & \cdots & \delta_{m\times n}(\tau)
	\end{array}\right]=\\
&= \left[\begin{array}{cccc}
	\psi_1(a)\cdot\delta_{1\times 1}(\tau)&\psi_1(a)\cdot\delta_{1\times 2}(\tau)&\cdots&\psi_1(a)\cdot\delta_{1\times n}(\tau) \\
	\psi_2(a)\cdot\delta_{2\times 1}(\tau)&\psi_2(a)\cdot\delta_{2\times 2}(\tau)&\cdots& \psi_2(a)\cdot\delta_{2\times n}(\tau)\\
	\vdots& \vdots& \ddots& \vdots\\
	\psi_{m}(a)\cdot\delta_{m\times 1}(\tau)& \psi_{m}(a)\cdot\delta_{m\times 2}(\tau)& \cdots & \psi_{m}(a)\cdot\delta_{m\times n}(\tau)
\end{array}\right]=\\
&= \left[\begin{array}{cccc}
	a*\delta_{1\times 1}(\tau)&a*\delta_{1\times 2}(\tau)&\cdots&a*\delta_{1\times n}(\tau) \\
	a*\delta_{2\times 1}(\tau)&a*\delta_{2\times 2}(\tau)&\cdots& a*\delta_{2\times n}(\tau)\\
	\vdots& \vdots& \ddots& \vdots\\
	a*\delta_{m\times 1}(\tau)& a*\delta_{m\times 2}(\tau)& \cdots & a*\delta_{m\times n}(\tau)
\end{array}\right].
	\end{align*}
Hence the isomorphism (\ref{eq:extension_isomorphism_product}) preserves the structure of $\F_q[t]-$modules. 
	Because $\rk\phi_i>\rk \psi_j$, then $\Ext^1_{\tau}(\phi_i,\psi_j)$ has a structure of  a $\mathbf t$-module. Let $\Pi_{i\times j}:\F_q[t]\lra {\mathrm{Mat}}_{n_i}(K)\{\tau\}$ be the map, which gives the structure of the $\mathbf t$-module on $\Ext^1_{\tau}(\phi_i,\psi_j)$ in the base $\tau^l, \,\, l=0,1,\dots, n_i-1$, same as in the proof of the Proposition \ref{prop:ext_as_t_module}.\\
	
	Let $e_{i\times j}^l$ denote the element $\tau^l\cdot E_{i\times j}$. 
	As $\F_q[t]-$modules
	$$\Ext^1_{\tau}\Big(\prod_{i=1}^n\phi_i, \prod_{j=1}^m\psi_j\Big)\cong\left[\begin{array}{ccc}
		\Ext^1_{\tau}(\phi_1,\psi_1)&\cdots&\Ext^1_{\tau}(\phi_n,\psi_1)\\
		\Ext^1_{\tau}(\phi_1,\psi_2)&\cdots&\Ext^1_{\tau}(\phi_n,\psi_2)\\
		\vdots& \ddots&\vdots\\
		\Ext^1_{\tau}(\phi_1,\psi_m)&\cdots&\Ext^1_{\tau}(\phi_n,\psi_m)\\
	\end{array}\right],$$
	we choose the following coordinate system: 
	\begin{align*}
		&\Big(e^l_{1\times 1}\Big)_{l=0}^{n_1-1}, \Big(e^l_{2\times 1}\Big)_{l=0}^{n_1-1},\dots,\Big(e^l_{m\times 1}\Big)_{l=0}^{n_1-1},\\
		&\Big(e^l_{1\times 2}\Big)_{l=0}^{n_2-1}, \Big(e^l_{2\times 2}\Big)_{l=0}^{n_2-1},\dots,\Big(e^l_{m\times 2}\Big)_{l=0}^{n_2-1},\\
		&\vdots\\
		&\Big(e^l_{1\times n}\Big)_{l=0}^{n_1-1}, \Big(e^l_{2\times n}\Big)_{l=0}^{n_1-1},\dots,\Big(e^l_{m\times n}\Big)_{l=0}^{n_n-1}.
	\end{align*}
Then the $\mathbf t-$module structure on $\Ext^1_{\tau}\Big(\prod_{i=1}^n\phi_i, \prod_{j=1}^m\psi_j\Big)$ is given by the map 
$\Pi:\F_q[t]\lra {\mathrm{Mat}}_{m\cdot{\sum}_{i=1}^nn_i}(K)\{\tau\}$, defined by 
\begin{align*}
	\Pi(t)=\left[\begin{array}{cccc}
		\Big[\Pi_{1\times j}(t)\Big]_{j=1}^m&0&\cdots&0 \\
		0& 	\Big[\Pi_{2\times j}(t)\Big]_{j=1}^m &\cdots& 0\\
		\vdots& \vdots& \ddots& \vdots\\
		0& 0& \cdots & 	\Big[\Pi_{n\times j}(t)\Big]_{j=1}^m
	\end{array}\right],
\end{align*}
where $\Big[\Pi_{i\times j}(t)\Big]_{j=1}^m$ is the diagonal matrix with the elements 
$\Pi_{i\times 1}(t)$, $\Pi_{i\times 2}(t)$, $\dots$, $\Pi_{i\times m}(t)$ on the diagonal. 
\end{proof}
	\begin{cor}\label{cor:product_ext_sequence}
	Assume that $\phi_i$, $\psi_j$ for $i=1,\dots, n$, $j=1,\dots, m$ are Drinfeld modules, such that  $\rk \phi_i>\rk\psi_j$ for all $i,j$. Then there is a short exact sequence of $\mathbf t$-modules 
	$$0\lra\Ext_{0,\tau}\Big(\prod_{i=1}^n\phi_i, \prod_{j=1}^m\psi_j\Big)\lra 
	\Ext^1_{\tau}\Big(\prod_{i=1}^n\phi_i, \prod_{j=1}^m\psi_j\Big) \lra \G_a^{n\cdot m}\lra 0. $$ 
	\end{cor}

\section{$\Ext^1_{\tau}(\phi,\psi)$ for Drinfeld modules $\phi$ and $\psi$ with $\rk \phi \leq \rk \psi$}\label{general}
\subsection{case $\rk\phi=\rk\psi$.}

This case is difficult to handle. We can express the generators $\delta^{(c\tau^k)},$ where $c\in K, \, k=0,1\dots ,$ of the space $\Derin(\phi, \psi)$ in the following way:
\begin{align}\label{iso2}
	\begin{split}
		\delta^{(c\tau^k)}(t)=c\tau^k\phi_t-\psi_t c\tau^k= c\tau^k\Big(\theta+\sum\limits_{i=1}^na_i\tau^i \Big)- \Big( \theta+\sum\limits_{j=1}^nb_j\tau^j \Big)c\tau^k\\
		= c(\up\theta k-\theta)\tau^k+\sum\limits_{j=k+1}^{n+k}\Big(c\up{a_{j-k}}k-\up c{j-k} b_{j-k}\Big)\tau^{j}.
		\end{split}
	\end{align}
Notice that  we cannot claim that  the polynomial in the variable $\tau$ has degree $n+k$ which was important for the identification of the biderivation $c_j\tau^j$ with the polynomial $w(\tau)\in K\{\tau\}$.  
However excluding finite number of $c's$ in every degree $k$ one can  express a biderivation $c_j\tau^j$ as  a polynomial 
in $K(\tau)$ of degree less than $n.$ 
At every   stage of the reduction process, in order to assure vanishing of the term with highest power of $\tau$, a solution of a polynomial equation with the coefficients in the function field is required. There are, in some cases, efficient algorithms to do this (cf. \cite{gs00}  ). However, it is unreasonable to expect that the roots of the polynomials that appear are expressed as polynomials in the variables $c_j.$ So, in general  we do not obtain a \tm module structure.

\subsection{case $\rk\phi < \rk\psi $.}

Let $K$ be a perfect field. Denote by $\sigma$ the inverse map to $\tau$. For the sake of simplicity we denote the value of $\sigma^k(c)$ as $\up c{-k}$ for $c\in K$. Then the ring $K\{\sigma\}$ is the ring of adjoint twisted polynomials in $K$ such that 
\begin{equation}\label{ad1}
	{\sigma}x=x^{(-1)}\sigma \quad \textnormal{for}\quad x\in K,
\end{equation}
(cf. \cite{g95}).  
We can consider the  \tsm module which we can define as in definition \ref{tmod} by replacing $\tau$ with $\sigma$. Similarly, we define a morphism of the \tsm modules.  Again the category of \tsm modules  with the zero  \tsm module attached is an additive, $\F_q[t]-$linear category. Moreover, 
there is an isomorphism of $\F_q[t]-$modules 
$$\Ext^1_\sigma(\Phi^{\mathrm{ad}}, \Psi^{\mathrm{ad}})\cong \Der(\Phi^{\mathrm{ad}}, \Psi^{\mathrm{ad}})/\Derin(\Phi^{\mathrm{ad}}, \Psi^{\mathrm{ad}}),$$
where $\Phi^{\mathrm{ad}}$  and $\Psi^{\mathrm{ad}}$ are \tsm modules.

Let ${\omega}(x)=\sum_{i=0}^na_ix^i$ then ${\omega}(\tau)\in K\{\tau\}$ and ${\omega}^{\mathrm{ad}}={\omega}(\sigma)\in 
K\{\sigma\}.$ 
Let $\phi(x)=\theta+\sum_{i=1}^na_ix^i$ and $\psi(x)=\theta+\sum_{i=1}^mb_ix^i$ be polynomials in $K[x]$ then 
$\phi(\tau), \psi(\tau)$ are Drinfeld modules and $\phi(\sigma), \psi(\sigma)$ are Drinfeld \tsm modules.
\begin{lem}\label{lemma:adjoint_t_module}
Let $\phi(x)=\theta+\sum_{i=1}^na_ix^i$ and $\psi(x)=\theta+\sum_{i=1}^mb_ix^i$ be polynomials in $K[x]$ such that
$\deg\phi(x)>\deg\psi(x)$ and let the \tm module $\Ext^1_\tau(\phi(\tau), \psi(\tau))$ be given by the map 
$$\Pi:\F_q[t]\lra \mathrm{Mat}_n(K\{\tau\}).$$
Then the $\F_q[t]-$module $\Ext^1_\sigma(\phi(\sigma), \psi(\sigma))$ has a structure of a \tsm module, given by the map
$$\widehat{\Pi}:\F_q[t]\lra \mathrm{Mat}_n(K\{\sigma\}),$$
where the matrix $\widehat{\Pi}_t$ is obtained from the matrix $\Pi_t$, by replacing $\tau$ with $\sigma$ and coefficients of the form $\up ci$ with coefficients of the form $\up c{-i}$. 
\end{lem}

\begin{proof}
	Similarly as in the proof of Lemma \ref{lemma_iso_Ext} we reduce a biderivation $\delta\in \Der(\phi(\sigma), \psi(\sigma))$ by the generators $\delta^{(c\sigma^k)}$ of $\Derin(\phi(\sigma), \psi(\sigma))$. As a result we obtain that  
	$$\Der(\phi(\sigma), \psi(\sigma))\cong K\{\sigma\}_{<n}:=\big\{w(\sigma)\in K\{\sigma\}\mid deg_\sigma(w)<n\big\}.$$
	Let $c\sigma^k$ be a generator of $K\{\sigma\}_{<n}$. Then 
	\begin{equation}\label{eq:multiplication_t_adjoint_on_generator}
		t*c_i\sigma^i=\psi(\sigma)c_i\sigma^i=\theta c_i\sigma^i+\sum\limits_{j=1}^mb_j\up {c_i}{-j}\sigma^{j+i}.
	\end{equation}
	Comparing the equality \eqref{eq:multiplication_t_adjoint_on_generator} with the value of $t*c_i\tau^i$ (see \eqref{eq:multiplication_t__on_generator}), we conclude that $t*c_i\sigma^i$ can be obtained from $t*c_i\tau^i$, by  replacing $\tau$ with $\sigma$ and coefficients of the form $\up {c_i}j$ with the coefficients of the form $\up {c_i}{-j}$.

	It is easy see, that in the same way, we can obtain the generator $\delta^{(c\sigma^k)}$ from $\delta^{(c\tau^k)}$. Therefore after the reduction process  $t*c_i\sigma^i$ can be obtained from  $t*c_i\tau^i$ by the previously described replacement.  This proves the claim. 
\end{proof}

\begin{ex}\label{ex:adjoint_t_module}
	 Consider the following polynomials: $\phi(x)=\theta+ax^3$ and $\psi(x)=\theta+bx^2$ for fixed $a,b\in K$. Then the  \tm module structure on  $\Ext^1_\tau(\phi(\tau), \psi(\tau))$ is given by the following matrix:
	$$\Pi_t=\begin{bmatrix}
		\theta&0&0\\
		0&\theta& \dfrac{b}{\up a1}\Big(\theta-\up\theta1\Big)\tau^2\\
		b\tau^2 & \dfrac{b\cdot \up b2}{\up a2}\tau^4 & \theta+ \dfrac{b\cdot \up b{2}\cdot \up b{4} }{\up a{5}\cdot \up a{2}}\tau^6
	\end{bmatrix}$$
	Then the \tsm module structure on $\Ext^1_\sigma(\phi(\sigma), \psi(\sigma))$ is given by the matrix:
	$$\widehat{\Pi}_t=\begin{bmatrix}
		\theta&0&0\\
		0&\theta& \dfrac{b}{\up a{-1}}\Big(\theta-\up\theta{-1}\Big)\sigma^2\\
		b\sigma^2 & \dfrac{b\cdot \up b{-2}}{\up a{-2}}\sigma^4 & \theta+ \dfrac{b\cdot \up b{-2}\cdot \up b{-4} }{\up a{-5}\cdot \up a{-2}}\sigma^6
	\end{bmatrix}$$
Let us note, that in this example we did not carry out calculations, because the matrix $\Pi_t$ can be easily obtained using the recursive procedure described in  Proposition \ref{prop:ext_as_t_module}.
\end{ex}

Consider the following maps:
\begin{equation}\label{maps1}
(-)^{\sigma}: K\{\tau\}\rightarrow K\{\sigma\};  \quad \Big(\sum_{i=0}^na_i\tau^i\Big)^{\sigma}=\sum_{i=0}^na_i^{(-i)}{\sigma}^i,
\end{equation}
\begin{equation}\label{maps2}
(-)^{\tau}: K\{\sigma\}\rightarrow K\{\tau\};  \quad \Big(\sum_{i=0}^nb_i\sigma^i\Big)^{\tau}=\sum_{i=0}^nb_i^{(i)}{\tau}^i.
\end{equation}
These maps are ${\mathbb F}_q$-linear mutual inverses.
 We associate with $\Phi : {\mathbb F}_q[t]\rightarrow {\mathrm{Mat}}_e(K\{\tau\})$ the adjoint homomorphism 
\begin{equation}\label{ad}
\Phi^{\sigma}:{\mathbb F}_q[t]\rightarrow {\mathrm{Mat}}_e(K\{\sigma\})
\end{equation}
such that    each matrix  $X_t$ is mapped  to $\Big[\big(X_t(\tau)\big)^{\sigma}\Big]^T$ i.e.
the matrix entry $X_{i,j}(\tau)$ is mapped  to $\big(X_{j,i}(\tau)\big)^{\sigma}.$ The inverse of $(-)^{\sigma}$ is given 
by the map that associates with  $\Gamma : {\mathbb F}_q[t]\rightarrow {\mathrm{Mat}}_e(K\{\sigma\})$ 
the following homomorphism:
\begin{equation}\label{ad1}
\Gamma^{\tau}:{\mathbb F}_q[t]\rightarrow {\mathrm{Mat}}_e(K\{\tau\}); \quad X_t \rightarrow \Big[\big(X_t(\sigma)\big)^{\tau}\Big]^T
\end{equation}
We have the following duality:
\begin{thm}\label{Duality}
	Assume that $K$ is a perfect $A-$field. 
Let $\Phi$ and $\Psi$ be \tm modules. Then there exists an isomorphism of ${\mathbb F}_q[t]$-modules:
\begin{equation}\label{diso}
\Ext^1_{\tau}(\Phi,\Psi)\cong \Ext^1_{\sigma}(\Psi^{\sigma},\Phi^{\sigma})
\end{equation} 
\end{thm} 
\begin{proof}
Notice that 
\begin{equation}\label{exta}
\Ext^1_{\tau}(\Phi,\Psi)=\Der(\Phi,\Psi)/\Der_{in}(\Phi,\Psi); \,\, t*\delta_t(\tau)=\Psi_t\cdot\delta_t(\tau)
\end{equation}
and
\begin{equation}\label{exta1}
\Ext^1_{\sigma}(\Psi^{\sigma},\Phi^{\sigma})=\Der(\Psi^{\sigma},\Phi^{\sigma})/\Der_{in}(\Psi^{\sigma},\Phi^{\sigma}); \,\, t*\delta_t(\tau)=\delta_t^{\sigma}(\sigma)\Psi_t^{\sigma}
\end{equation}
One readily verifies that $(-)^{\sigma}$ is well defined since it maps inner biderivations onto inner biderivations. Similarly,
for $(-)^{\tau}.$ Thus we have a bijective map (induced by $(-)^{\sigma}$) $\Ext^1_{\tau}(\Phi,\Psi)\rightarrow \Ext^1_{\sigma}(\Psi^{\sigma},\Phi^{\sigma}).$
As $$(t*c\tau^k)^{\sigma}=(\Psi_tc\tau^k)^{\sigma}=c^{(-k)}\sigma^k\Psi_t^{\sigma}=t*c^{(-k)}\sigma^k=t*(c\tau^k)^{\sigma}
$$ we see that $(-)^{\sigma}$ is a homomorphism of ${\mathbb F}_q[t]$-modules.
\end{proof}

	 Lemma \ref{lemma:adjoint_t_module} and  Theorem \ref{Duality} allow one  to compute easily the ${\mathbb F}_q[t]$-module structure of $\Ext^1_{\tau}(\phi, \psi)$   
	for Drinfeld modules satisfying $\rk\phi<\rk\psi.$ We will see that this ${\mathbb F}_q[t]$-module structure
	comes from the \tsm module structure in a natural way. 
\begin{ex}
	Let $\psi_t=\theta +b\tau^2$ and $\phi_t=\theta +a\tau^3,$  for fixed $a,b\in K,$ be  Drinfeld modules. 
	Then    Theorem \ref{Duality} implies that there is an isomorphism of $\F_q[t]-$modules
	$$\Ext^1_{\tau}(\theta +b\tau^2,\theta +a\tau^3)\cong \Ext^1_{\sigma}(\theta+\up a{-3}\sigma^3,\theta+\up b{-2}\sigma^2).$$
	By Lemma \ref{lemma:adjoint_t_module} the adjoint  ${\mathbf t}^{\sigma}-$module structure on  
	$$\Ext^1_{\sigma}(\theta+\up a{-3}\sigma^3,\theta+\up b{-2}\sigma^2)$$ is given by the matrix $\widehat{\Pi}_t$
	obtained from matrix $\Pi_t$ of 
	$$\Ext^1_{\tau}(\theta+\up a{-3}\tau^3,\theta+\up b{-2}\tau^2),$$
	 by replacing $\tau$ with $\sigma$ and coefficients of the form $\up ci$ with coefficients of the form $\up c{-i}$. 
	Therefore $\Ext^1_{\tau}(\theta +b\tau^2,\theta +a\tau^3)$ has the \tsm module structure, given by the matrix
	$$\begin{bmatrix}
		\theta&0&0\\
		0&\theta& \dfrac{\up b{-2}}{\up a{-4}}\Big(\theta-\up\theta{-1}\Big)\sigma^2\\
		\up b{-2}\sigma^2 & \dfrac{\up b{-2}\cdot \up b{-4}}{\up a{-5}}\sigma^4 & \theta+ \dfrac{\up b{-2}\cdot \up b{-4}\cdot \up b{-6} }{\up a{-8}\cdot \up a{-5}}\sigma^6
	\end{bmatrix}$$
\end{ex}

	The following theorem is a consequence of  Theorem \ref{thm:short_sequence_t_modules},  Lemma \ref{lemma:adjoint_t_module} and  Theorem \ref{Duality}.
\begin{thm}  Assume that $K$ is a perfect $A-$field. 
	Let $\phi$ and $\psi$ be Drinfeld modules, such that $\rk\phi<\rk\psi$. There is a short exact sequence of \tsm modules
	\begin{equation}
		0\lra \Ext^1_{0,\tau}(\phi, \psi)\lra \Ext^1_\tau(\phi,\psi)\lra \G_a\lra 0.
	\end{equation}
\end{thm}
\begin{rem}
	In a similar way, one can prove ''\tsm version'' of the Theorem \ref{thm:product_ext} and Corollary \ref{cor:product_ext_sequence}  for the products of Drinfeld modules.  
\end{rem}

\section{$\Ext^1_\tau(\Phi, \psi)$, where $\Phi$ is a \tm module and $\psi$ is a Drinfeld module}\label{Phipsi}

In this section we consider $\Ext^1_{\tau}$, where $\Phi$ is a \tm module such that the matrix at the highest power  $\tau^{\rk \Phi}$ of ${\Phi}_t$ is invertible and $\psi$ is a Drinfeld module. First  assume that this matrix is the identity matrix.
 
\begin{lem}\label{lem:redukcja_t_modul_i_Drinfeld}
	Let $\Phi_t=(\theta I+ N_{\Phi})\tau^0+\sum\limits_{j=1}^n A_j\tau^j $ be a \tm module of dimension $d$, where $A_n=I$  and let  $\psi_t=\theta +\sum\limits_{j=1}^mb_j\tau^j$ be a Drinfeld module. If $\rk \Phi>\rk\psi$ then we have the 
	following isomorphisms of  $\F_q-$ vector spaces :
	\begin{itemize}
		\item[(i)] $\Ext^1_\tau(\Phi, \psi)\cong  \Big(K\{\tau\}_{<\rk\Phi}\Big)^{\oplus d}$  
		\item[(ii)] $\Ext_{0,\tau}(\Phi, \psi)\cong \bigoplus_{i=1}^dK\{\tau\}_{ [1,r_i] }$, where 
		$r_i=\rk\Phi-1$, if the  $i-$th row of  $N_{\Phi}$ is null and $r_i=\rk\Phi$ otherwise.  
	\end{itemize}
\end{lem}
\begin{proof}

Part $(i)$: 
	Any biderivation $\delta\in\Der(\Phi,\psi)$ is described by a matrix  belonging to $\mathrm{Mat}_{1\times d}(K\{\tau\})$.  
Let $E_i\in\mathrm{Mat}_{1\times d}(K)$ be a matrix which has $1$ at the $i$-th place and zeroes otherwise.
We will determine  an inner biderivation  $\delta^{(c\tau^k E_i)}$. 
Denote by  $a_{j, i\times l}$ the  $i\times l$-term of matrix $A_j$. 
Then the corresponding inner biderivation has the following form:

	\begin{align}\label{eq:biderywacja_wewnetrzna_t-mod_i_Drinfeld}
		\delta^{(c\tau^k E_i)}_t=& c\tau^k E_i\Big( (\theta I+ N_{\Phi})\tau^0+\sum\limits_{j=1}^n A_j\tau^j \Big)-\Big(\theta +\sum\limits_{j=1}^mb_j\tau^j\Big)c\tau^k E_i=\\
		=&\Bigg[\sum_{j=1}^{n-1}a_{j,i\times 1}^{(k)}\tau^{k+j},\dots, \sum_{j=1}^{n-1}a_{j,i\times i-1}^{(k)}\tau^{k+j}, 
		c(\theta^{(k)}-\theta)\tau^k+ \nonumber \\
		&+  \sum_{j=1}^{m}\Big(a_{j,i\times i}^{(k)}-c^{(j)}\Big)\tau^{k+j}
		+\sum_{j=m+1}^{n-1}a_{j,i\times i}^{(k)}\tau^{k+j} +c\tau^{k+n}, \nonumber  \\
		&\quad \sum_{j=1}^{n-1}a_{j,i\times i+1}^{(k)}\tau^{k+j},\dots, 
		\sum_{j=1}^{n-1}a_{j,i\times d}^{(k)}\tau^k+j  \Bigg]+cE_iN_{\Phi}^{(k)}\tau^k, \nonumber 		 
	\end{align} 
where $N_{\Phi}^{(k)}$ indicates that all terms of the matrix  $N_{\Phi}$ are raised to the power $q^k$. 
	Notice that a  polynomial at  the $i$-th  coordinate of the inner biderivation  \eqref{eq:biderywacja_wewnetrzna_t-mod_i_Drinfeld} has degree $k+n$ and the polynomials at all other coordinates have degrees less than $k+n.$ 
	Thus we can  proceed similarly as in the proof of Lemma \ref{lemma_iso_Ext} and reduce the biderivation $\delta$ so that at every coordinate we obtain a polynomial in $\tau$ of degree less than  $n=\rk\Phi$. 
	It is obvious that two different reduced biderivations determine different cosets in
	 $\Der(\Phi,\psi)/\Derin(\Phi,\psi).$ This proves part $(i)$.\\
	Part $(ii)$: Let $\delta\in\Der_0(\Phi,\psi)$, i.e. $\delta_t$ has a zero constant term.
	Notice that by means of the inner biderivation 
	  $\delta^{(c\tau^0E_i)}\in\Der_0(\Phi,\psi)$ for $i\in\{0,1,\dots, d\}$ we can reduce the $i-$th coordinate of  $\delta_t$ 
	  to a polynomial of degree less than $\rk\Phi$. If   $\delta^{(c\tau^0E_j)}\notin\Der_0(\Phi,\psi)$ then the $j-$th  coordinate of $\delta_t$ can be reduced to a polynomial of degree less than or equal to 
	   $\rk\Phi$. To finish the proof notice that the form of the inner biderivation 
	   \eqref{eq:biderywacja_wewnetrzna_t-mod_i_Drinfeld}
	 implies that $\delta^{(c\tau^0 E_i)}\in\Der_0(\Phi,\psi)$ iff
		$E_iN_{\Phi}=0$ i.e. the $i-$th row of  $N_{\Phi}$ is zero.
\end{proof}
\begin{prop}\label{propI}
		Let $\Phi_t=(\theta I+ N_{\Phi})\tau^0+\sum\limits_{j=1}^n A_j\tau^j $ be a \tm module of dimension $d$, where $A_n=I$  and let  $\psi_t=\theta +\sum\limits_{j=1}^mb_j\tau^j$ be a Drinfeld module. If $\rk \Phi>\rk\psi$ then 
		\begin{itemize}
			\item[$(i)$] $\Ext^1_{\tau}(\Phi,\psi)$ has the natural structure of a  \tm module,
			\item[$(ii)$] there exists   a short exact sequence of  \tm modules:
			$$0\lra \Ext_{0,\tau}(\Phi,\psi)\lra \Ext^1_{\tau}(\Phi,\psi)\lra \G_a^{s}\lra0,$$
			where $s$  is the number of zero rows of the nilpotent matrix $N_{\Phi}$.  
		\end{itemize}
\end{prop}
\begin{proof}\label{prop:ext_dla_t_mod_i_Drinfeld}
	Proof of part $(i)$ is analogous to the proof of Proposition \ref{prop:ext_as_t_module}.
	Let $E_i\in {\mathrm{Mat_{1\times d}}}$ be a matrix which has $1$ at the $i-$th place and zeroes at all other places.  From Lemma \ref{lem:redukcja_t_modul_i_Drinfeld} we know that 
	$$\Ext^1_\tau(\Phi, \psi)\cong 
	\Big(K\{\tau\}_{<\rk\Phi}\Big)^{\oplus d}.$$
	Thus in order to transfer the $\F_q[t]-$module structure to the space \linebreak $\Big(K\{\tau\}_{<\rk\Phi}\Big)^{\oplus d}$ 
	it suffices to find multiplication by  $t$ on the generators of the form  $c\tau^kE_i$, where $0\leq k<\rk \Phi$ and $i=1,2,\dots, d$.  If a degree of  $t*c_k\tau^k$ is bigger than $\rk\Phi$ then appropriate terms, starting from the term with the highest power of  $\tau,$ can be reduced 
	by means of the inner biderivations of the form  $\delta^{(c\tau^kE_i)}$. 
	After each reduction at every coordinate we obtain a polynomial satisfying conditions
	 $(i)-(iv)$ of the proof of Proposition  \ref{prop:ext_as_t_module}. 
	 In this way after the following choice of basis: 
	$$(\tau^kE_1)_{k=0}^{n-1}, (\tau^kE_2)_{k=0}^{n-1},\dots, (\tau^kE_d)_{k=0}^{n-1}$$
	we see that $\Ext^1_{\tau}(\Phi,\psi)$ is a \tm module with a zero nilpotent matrix.  
	\\
	%%%%
	It is also worth noting that reducing by the inner biderivations 
	$\delta^{(c\tau^k E_i)}$ for any $i$  we cannot obtain a nonzero term at $\tau^0E_{l}$ if $\delta^{(c\tau^0E_l)}\in\Der_0(\Phi,\psi)$. This means that if  $\delta^{(c\tau^0E_l)}\in\Der_0(\Phi,\psi)$ then the row corresponding to the coordinate $\tau^0E_l$ in the matrix  $\Pi_t$ is of the form
\begin{equation}\label{star}
\quad\Big[0,\dots,0,\theta,0\dots, 0 \Big],
\end{equation}
	where the only element $\theta$ corresponds to the coordinate $\tau^0E_l$.
	\\
	%%%%
	 For the proof of $(ii)$ recall that there exists a canonical embedding \linebreak
	  $\Ext_{0,\tau}(\Phi,\psi)\lra \Ext^1_{\tau}(\Phi,\psi)$ defined on the level of biderivations.
	  This map is given by a matrix of twisted polynomials in $K\{\tau\}$ such that the row corresponding to the coordinate
	   $\tau^0E_i$ is zero iff the inner biderivation 
	     $\delta^{(c\tau^0E_i)}\in\Der_0(\Phi,\psi)$. Let $s$ be the number of inner biderivations  
	     of the form $\delta^{(c\tau^0E_i)}$ belonging to $\Der_0(\Phi,\psi)$.
	Consider  a map $g: \Ext_{\tau}^1(\Phi,\psi)\lra \G_a^{s}$ given by a row matrix where the only nonzero elements are 
	equal to $1$ at places corresponding to the coordinates for which
	 $\delta^{(c\tau^0E_i)}\in\Der_0(\Phi,\psi)$. 
	 In order to show that $g$ is a morphism of \tm modules one has to check the equality $g\cdot \Pi_t=\theta \cdot g.$
	 This follows from the fact that the row of  $\Pi_t$ corresponding to the coordinate $\tau^0E_l$ is of the form (\ref{star}) if $\delta^{(c\tau^0E_l)}\in\Der_0(\Phi,\psi)$.
	Thus we obtain the following exact sequence of \tm modules. 
	$$0\lra \Ext_{0,\tau}(\Phi,\psi)\lra \Ext^1_{\tau}(\Phi,\psi)\lra \G_a^{s}\lra0.$$
\end{proof}

The case where  $\Phi$ is a  \tm module such that   ${\Phi}_t $ has an invertible matrix at $\tau^{\rk\Phi}$ 
can be derived from the Proposition \ref{propI}.
 We have the following: 
\begin{thm}\label{thm:Ext_dla_t_modul_i_Drinfeld}
	Let $\Phi_t=(\theta I+ N_{\Phi})\tau^0+\sum\limits_{j=1}^n A_j\tau^j $ be a  \tm module of dimension $d$, where $A_n$ is an invertible matrix  and let  $\psi_t=\theta +\sum\limits_{j=1}^mb_j\tau^j$ be a Drinfeld module. If $\rk \Phi>\rk\psi$ then 
	\begin{itemize}
		\item[$(i)$] $\Ext^1_{\tau}(\Phi,\psi)$ has a natural structure of a \tm module,
		\item[$(ii)$] there exists a short exact sequence of  \tm modules
		$$0\lra \Ext_{0,\tau}(\Phi,\psi)\lra \Ext^1_{\tau}(\Phi,\psi)\lra \G_a^{s}\lra0,$$
		where $s$ is the number of zero rows in the matrix $A_n^{-1}N_{\Phi}$.  
	\end{itemize}
\end{thm}
\begin{proof}

In the current situation, for the reduction process, we use different inner biderivations than these used in the proof of  Proposition \ref{propI}  .
Let $c\tau^kE_iA_n^{-1} \in \mathrm{Mat}_{1\times d}(K\{\tau\}).$   We see that  the inner biderivations  $\delta^{(c\tau^kE_iA_n^{-1} )}$  again have the property that a polynomial at the $i$-th coordinate  has degree  $n+k$ 
and  polynomials at all other coordinates  have degrees less than $n+k.$  
Analogous  reasoning to that of   Proposition \ref{propI} completes the proof of $(i)$. \\
	For the proof of  $(ii)$ we have to check that the inner  biderivation  $\delta^{(c\tau^0E_iA_n^{-1} )}$ belongs to  $\Der_0(\Phi, \psi)$. To do this we compute a constant term of this inner biderivation: 
		\begin{align*}
		c_0&\tau^0E_iA_n^{-1}(\theta I+N_{\Phi})-\theta c_0\tau^0E_iA_n^{-1}=\\
		=&c_0\tau^0E_iA_n^{-1}\theta I+c_0\tau^0E_iA_n^{-1}N_{\Phi}-\theta c_0\tau^0E_iA_n^{-1}=\\
		=&c_0\tau^0\theta IE_iA_n^{-1}-\theta c_0\tau^0E_iA_n^{-1} +c_0\tau^0E_iA_n^{-1}N_{\Phi}=c_0\tau^0E_iA_n^{-1}N_{\Phi}\\
	\end{align*}
Therefore, $\delta^{(c\tau^0E_iA_n^{-1} )}\in\Der_0(\Phi, \psi) $ iff the $i-$th row of the matrix  $A_n^{-1}N_{\Phi}$ is zero. 
Similar reasoning as in the proof of Proposition \ref{propI} yields the result.

\end{proof}

Just as in Theorem \ref{thm:product_ext}, one can show the following:
\begin{thm}\label{gen2}
	Let $\Phi_t=(\theta I+ N_{\Phi})\tau^0+\sum\limits_{j=1}^n A_j\tau^j $ be a \tm module of dimension $d$, where $A_n$ 
	is an invertible matrix. Let   $\psi_i$  $i=1,\dots, m$ be Drinfeld modules. If $\rk \Phi>\rk\psi_i$ for $i=1,\dots,m$ then 
	\begin{itemize}
		\item[$(i)$] $\Ext^1_{\tau}(\Phi,\prod_{i=1}^{m}\psi_i)$ has a natural  structure of a \tm module,
		\item[$(ii)$] there exists a short exact sequence of \tm modules
		$$0\lra \Ext_{0,\tau}(\Phi,\prod_{i=1}^m\psi_i)\lra \Ext^1_{\tau}(\Phi,\prod_{i=1}^m\psi_i)\lra \G_a^{m\cdot s}\lra0,$$
		where $s$ is the number of zero rows of the matrix $A_n^{-1}N_{\Phi}$.  
	\end{itemize}
\end{thm}
One can also show ''\tsm version'' of Theorems  \ref{thm:Ext_dla_t_modul_i_Drinfeld} and \ref{gen2}.
We state a more general  result for the product of Drinfeld modules:
\begin{thm}
	Let $K$ be a perfect field. Let $\Psi_t=(\theta I+ N_{\Psi})\tau^0+\sum\limits_{j=1}^n B_j\tau^j $ be a \tm module of dimension $d$, where $B_n$ is an invertible matrix  and let  $\prod_{i=1}^m\phi_i$ be a product of  Drinfeld modules. If $\rk \Psi>\rk\phi_i$ for $i=1,\dots,m$ then 
	there exists a short exact sequence of \tsm modules:
	$$0\lra \Ext_{0,\tau}(\prod_{i=1}^m\phi_i,\Psi)\lra \Ext^1_{\tau}(\prod_{i=1}^m\phi_i,\Psi)\lra \G_a^{m\cdot s}\lra0,$$
	where $s$ is the number of nonzero rows of the matrix $A_n^{-1}N_{\Psi}$. 
\end{thm}

\section{$\Ext^1_\tau(\Phi, C^{\otimes e})$ for   \tm module $\Phi$}\label{Carlitz}

In this section, we show that the method of reductions by means of inner biderivations used earlier allows one to determine 
the space of extensions for  a \tm module $\Phi$ such that $\Phi_t$ has an invertible matrix at
 $\tau^{\rk\Phi}$ and $C^{\otimes e}$ is the $e-$th tensor of a  Carlitz module.\\
Recall from \cite{AT90} that $C^{\otimes e}=\theta I_e+N_e+E_{e\times 1}\tau$ where
\begin{align*}
	I_e&=\left[\begin{array}{cccc}
		1& 0 & \cdots & 0\\
		0& 1 & \cdots & 0\\
		\vdots& \vdots & \ddots & \vdots\\
		0& 0 & \cdots & 1
	\end{array} \right],\quad
	N_e=\left[\begin{array}{ccccc}
		0& 1 & 0&\cdots & 0\\
		0& 0 & 1& \cdots & 0\\
		\vdots& \vdots &\vdots & \ddots & \vdots\\
		0& 0 &0  &\cdots & 1\\
		0& 0 &0  &\cdots & 0
	\end{array} \right],\\
E_{e\times 1}&=\left[\begin{array}{cccc}
	0& 0 & \cdots & 0\\
	\vdots& \vdots & \ddots & \vdots\\
	0& 0 & \cdots & 0\\
	1& 0 & \cdots & 0 \end{array} \right]\in\mathrm{Mat}_{e}(K). 
\end{align*}
\begin{lem}\label{lem:ext_t_mod_i_tensor_Carlitza}
	Let $\Phi_t=(\theta I+ N_{\Phi})\tau^0+\sum\limits_{j=1}^n A_j\tau^j $ be a \tm module of dimension $d$, such 
	that
	 $A_n$ is an invertible matrix and let   $C^{\otimes e}$ be the $e-$th tensor of the Carlitz module.  If $\rk \Phi\geq 2$ then we have the following  isomorphisms of $\F_q-$vector spaces:
\begin{itemize}
	\item[(i)] $\Ext^1_\tau(\Phi, C^{\otimes e})\cong  \mathrm{Mat}_{e\times d}\Big(K\{\tau\}_{<\rk\Phi}\Big),$  
	\item[(ii)] $\Ext_{0,\tau}(\Phi, C^{\otimes e})\cong \left[\begin{array}{cccc}
		K\{\tau\}_{ [1,r_{1}] } & K\{\tau\}_{ [1,r_{2}] }& \cdots & K\{\tau\}_{ [1,r_{d}]}\\
		K\{\tau\}_{ [1,\rk\Phi] } & K\{\tau\}_{ [1,\rk\Phi] }& \cdots & K\{\tau\}_{ [1,\rk\Phi]}\\
		\vdots & \vdots & \ddots & \vdots \\
		K\{\tau\}_{ [1,\rk\Phi] } & K\{\tau\}_{ [1,\rk\Phi] }& \cdots & K\{\tau\}_{ [1,\rk\Phi]}\\
	\end{array} \right]$, \\
	
	where
	$r_{j}=\rk\Phi-1$ if the  $j-$th row of the matrix $A_n^{-1}N_{\Phi}$ is zero and $r_j=\rk\Phi$ otherwise.  
\end{itemize}
\end{lem}
\begin{proof}
	Let $E_{i\times j}$ be a matrix with the  one nonzero entry equal to $1$ at the place  $i\times j$. 
	Then a similar calculation to that done in the proof of  the Theorem \ref{thm:Ext_dla_t_modul_i_Drinfeld} 
	shows that the inner biderivation
	 $\delta^{(U)}$ for $U=c_k\tau^kE_{i\times j}A_n^{-1}$  is given by a matrix which as the $i\times j$ entry has
	 a polynomial of degree
	  $k+\rk\Phi$  and at all the other entries polynomials of degrees  less than $k+\rk\Phi$. 
	  Reduction by  biderivations chosen in such a way yields 
	   $(i)$.\\
 For the proof of part (ii) notice that if the inner biderivations $\delta^{(U)}$ for $U=c_0\tau^0E_{i\times j}A_n^{-1}$  
 belong to  $\Der_0(\Phi,C^{\otimes e} )$ then we can use this biderivation in the reduction process and get  $r_{i\times j}=\rk\Phi -1$  and 
	$r_{i\times j}=\rk\Phi$ otherwise. $r_{i\times j}$ denotes here the degree of a polynomial occuring as the $i\times j $
	entry of the reduced matrix. In order to decide whether 
	$\delta^{(U)}\in \Der_0(\Phi,C^{\otimes e} )$ or not one has to compute the constant term of $\delta^{(U)}$ where  $U=c_0\tau^0E_{i\times j}A_n^{-1}.$   This constant term is of the following form: 
		\begin{align*}
		c_0&\tau^0E_{i\times j}A_n^{-1}(\theta I_d+N_{\Phi})- (\theta I_e+N_e)c_0\tau^0E_{i\times j}A_n^{-1}=\\
		=&c_0\theta\tau^0E_{i\times j}A_n^{-1} +c_0\tau^0E_{i\times j}A_n^{-1}N_{\Phi}-c_0\theta\tau^0E_{i\times j}A_n^{-1} - 
		c_0\tau^0N_eE_{i\times j}A_n^{-1}=\\
		=&c_0\tau^0E_{i\times j}A_n^{-1}N_{\Phi}-c_0\tau^0N_eE_{i\times j}A_n^{-1}=
		c_0\Big(E_{i\times j}A_n^{-1}N_{\Phi}-N_eE_{i\times j}A_n^{-1}\Big)\tau^0\\
	\end{align*}
Since $N_e=\sum_{k=2}^e E_{k-1\times k}$ we have 
$$N_eE_{i\times j}=\left\{\begin{array}{ccc}
	0 & \textnormal{if} &i=1\\
	E_{i-1\times j}& \textnormal{if} &i=2,3,\dots, e.\\
\end{array}\right.
$$
Then if  $i=2,3,\dots, e$ we have the equality $N_eE_{i\times j}A_n^{-1}=E_{i-1\times j}A_n^{-1}\neq 0$. This is because as the result of the matrix multiplication $E_{i-1\times j}\cdot A_n^{-1}$ we obtain a matrix in which the  $(i-1)$-th 
row is equal to the $j$-th row of the matrix $A_n$ and all other rows are zero.
Of course all rows of $A_n$ are nonzero.  Therefore one concludes that  $\delta^{(U)}\notin \Der_0(\Phi,C^{\otimes e} )$ for $U=c_0\tau^0E_{i\times j}A_n^{-1}$ and $i=2,3,\dots, e$. For $i=1$ we see that $\delta^{(U)}\in \Der_0(\Phi,C^{\otimes e} )$
 iff the $j-$th row of the matrix $A_n^{-1}N_{\Phi}$ is zero. 
\end{proof}
\begin{thm}\label{crl}
		Let $\Phi_t=(\theta I+ N_{\Phi})\tau^0+\sum\limits_{j=1}^n A_j\tau^j $ be a \tm module of dimension $d$ where $A_n$ is an invertible matrix   and let  $C^{\otimes e}$ be the  $e-$th tensor of the  Carlitz module. If $\rk \Phi\geq 2$, then 
		\begin{itemize}
			\item[$(i)$] $\Ext^1_{\tau}(\Phi,C^{\otimes e})$ has a natural structure of a  \tm module,
			\item[$(ii)$] there exists a short exact sequence of  \tm modules
			$$0\lra \Ext_{0,\tau}(\Phi,C^{\otimes e})\lra \Ext^1_{\tau}(\Phi,C^{\otimes e} )\lra \G_a^{s}\lra0,$$
			where $s$ is the number of nonzero rows of the matrix $A_n^{-1}N_{\Phi}$.  
		\end{itemize}
\end{thm}
\begin{proof}
The scheme of the proof is the same as of analogous theorems. One uses Lemma 
 \ref{lem:ext_t_mod_i_tensor_Carlitza}  and reduction by means of the inner biderivations described in the proof of this Lemma. The reduction process fulfills the properties  $(i)-(iv)$ described in the proof of the Proposition \ref{prop:ext_as_t_module}. In the current situation we use the following coordinate system: 
	$$\begin{array}{cccc}
		\Big(E_{1\times 1} c_k\tau^k\Big)_{k=0}^{\rk\Phi-1}, & \Big(E_{1\times 2} c_k\tau^k\Big)_{k=0}^{\rk\Phi-1},&\cdots & \Big(E_{1\times d} c_k\tau^k\Big)_{k=0}^{\rk\Phi-1},\\\\
		\Big(E_{2\times 1} c_k\tau^k\Big)_{k=0}^{\rk\Phi-1}, & \Big(E_{2\times 2} c_k\tau^k\Big)_{k=0}^{\rk\Phi-1},&\cdots & \Big(E_{2\times d} c_k\tau^k\Big)_{k=0}^{\rk\Phi-1},\\
		\vdots & \vdots & \vdots & \vdots \\
		\vdots & \vdots & \vdots & \vdots \\
		\Big(E_{e\times 1} c_k\tau^k\Big)_{k=0}^{\rk\Phi-1}, & \Big(E_{e\times 2} c_k\tau^k\Big)_{k=0}^{\rk\Phi-1},&\cdots & \Big(E_{e\times d} c_k\tau^k\Big)_{k=0}^{\rk\Phi-1}.\\
	\end{array}$$
	It is worth pointing out a significant difference which occurs in this case. In the former cases  $\Ext^1_{\tau}$ was a  \tm module with the zero nilpotent matrix which was a result of the aforementioned properties  $(iii)$ and $(iv)$  of Proposition \ref{prop:ext_as_t_module}. In the current case, reductions again will not change this, but at the stage of the multiplication
	 $t*E_{i\times j}c_k\tau^k$ we can obtain nonzero entries of the nilpotent matrix. More precisely for $i=2,3,\dots, e$ we obtain:
	\begin{align*}
		t*E_{i\times j}c_k\tau^k&=C^{\otimes e}\cdot E_{i\times j}c_k\tau^k=\Big(\theta I_e+N_e+E_{1\times e}\tau\Big)E_{i\times j}c_k\tau^k\\
		&=\theta E_{i\times j}c_k\tau^k+ E_{i-1\times j}c_k\tau^k\\
		&=\Big[0,\dots,0,\podwzorem{0,\dots,0,1,0\dots,0}{i-1\times j}, \podwzorem{0,\dots,0,\theta,0\dots,0}{i\times j},0\dots, 0  \Big]_{\mid c_k}.
	\end{align*}  
Thus  $N_{\Ext^1(\Phi, C^{\otimes e})}$  is an upper triangular matrix with zeroes on the diagonal and therefore nilpotent. 
This finishes the proof of  $(i)$. Proof of $(ii)$ follows the lines of the proof of  part $(ii)$ of  Theorem \ref{thm:Ext_dla_t_modul_i_Drinfeld}.  
\end{proof}

\begin{rem}
Similarly as at the end of Section   \ref{Phipsi}  using  Theorem \ref{Duality} we can prove the "\tsm - version" of  the Theorem \ref{crl}.   
\end{rem}

\section{Pushouts and pullbacks for \tm modules and their applications}\label{push}

\newcommand{\wlozenie}{\left[\begin{array}{c}
		0\\1
	\end{array}\right]}
\newcommand{\rzut}{\left[\begin{array}{cc}
		1&0
	\end{array}\right]}
In this section we  study pullbacks and pushouts in the category of  \tm modules. For definitions and basic properties of these the reader is advised to consult \cite{Mac1}. In what follows
 by $\wlozenie$ we mean the map of  \tm modules $(\psi, G^m_a)\hookrightarrow (X, G^{m+n}_a)$ which
 on underlying group
  schemes is given  by the  injection $G^m_a\cong 0^n\times G^{m}_a\hookrightarrow G^{m+n}_a.$ 
  Similarly, by $\rzut$ we denote the map of  \tm modules $(X, G^{m+n}_a)\rightarrow (\psi, G^n_a)$ which
 on underlying group
  schemes is given  by the  surjection $G^{m+n}_a \rightarrow G^n_a \times 0^m\cong G^n_a .$ 
  So that in the category of ${\mathbb F}_q[t]$-modules pullbacks and pushouts exist follows from the fact that this category is abelian. In the next theorem we show that they exist   in the category of \tm modules and can be nicely described in the language of biderivations.

\begin{thm}\label{thm:pull-back_push-out}
Let  
$$\delta: \quad 0\lra F\lra X\lra E\lra 0$$
be a short exact sequence
of \tm modules, given by the biderivation $\delta$.
Then
\begin{itemize}
	\item[$(i)$] for each morphism of \tm modules $g:G\lra E$, the pull-back of $\delta$ by $g$ is a \tm module, given by the biderivation $\delta\cdot g,$
	\item[$(ii)$] for each morphism of \tm modules $f:F\lra G$, the push-out of $\delta$ by $f$ is a \tm module, given by the biderivation $f\cdot\delta$. 
\end{itemize}
\end{thm}	
 \begin{proof}
 	The sequence $\delta$ has the following form: 
 	\begin{equation}
	\delta: \quad 0 \lra F\uplra{\wlozenie} X\uplra{\rzut} E\lra 0,
	\end{equation}
 	where $F$  (resp.  $E$)  is given by the map $\Psi:\F_q[t]\lra {\mathrm{Mat}}_{e}(K\{\tau\})$ (resp. $\Phi:\F_q[t]\lra {\mathrm{Mat}}_{d}(K\{\tau\})$) and $X$ is given by the following block matrix:
 	$$\left[\begin{array}{cc}
 		\Phi & 0\\
 		\delta & \Psi
 	\end{array}\right]:\F_q[t]\lra {\mathrm{Mat}}_{e+d}(K\{\tau\}).$$
 {\em Proof of part $(i):$} Let $G$ be given by $\Xi:\F_q[t]\lra {\mathrm{Mat}}_{r}(K\{\tau\})$. We claim that the following diagram
 $$\xymatrix @C+2pc @R+1.5pc{0\ar[r] & F\ar[r]^{\wlozenie} \ar[d]^{=} & Y \ar[r]^{\rzut} \ar[d]^{\left[\begin{array}{cc}
 			g&0\\
 			0&1
 		\end{array}\right]} & G \ar[r] \ar[d]^{g}& 0\\
 0\ar[r] & F\ar[r]_{\wlozenie} & X \ar[r]_{\rzut} & E \ar[r] & 0}$$
 is commutative with exact rows, where $Y$ is given by the map: 
 	$$\left[\begin{array}{cc}
 	\Xi & 0\\
 	\delta\cdot g & \Psi
 \end{array}\right]:\F_q[t]\lra {\mathrm{Mat}}_{e+d}(K\{\tau\}).$$
It is obvious, that the rows are exact, and it is easy to see, that the two squares are commutative.  
We will check that the middle vertical map is a morphism of  \tm modules. 
\begin{align*}
\left[\begin{array}{cc}
	g&0\\
	0&1
\end{array}\right]\cdot	\left[\begin{array}{cc}
		\Xi & 0\\
		\delta\cdot g & \Psi
	\end{array}\right]=
\left[\begin{array}{cc}
	g\Xi & 0\\
	\delta\cdot g & \Psi
\end{array}\right]=
\left[\begin{array}{cc}
	\Phi & 0\\
	\delta & \Psi
\end{array}\right]\cdot 
\left[\begin{array}{cc}
	g&0\\
	0&1
\end{array}\right], 
\end{align*}
where $g\Xi=\Phi g$ because $g$ is a morphism of \tm modules.  

It remains to show that the universal property for a pullback holds true. 
So, assume that there is a \tm module $\widehat{Y}$ with the morphisms of \tm modules $\alpha:\widehat{Y}\lra G$ and $\beta=\left[\begin{array}{c}
	\beta_1\\
	\beta_2
\end{array}\right]:\widehat{Y}\lra X$ such that
$g\alpha=\rzut
\left[\begin{array}{c}
	\beta_1\\
	\beta_2
\end{array}\right].$ 
We claim that there is a morphism of \tm modules $\gamma=\left[\begin{array}{c}
	\gamma_1\\
	\gamma_2
\end{array}\right]:\widehat{Y}\lra Y$ such that 
$$\rzut\left[\begin{array}{c}
	\gamma_1\\
	\gamma_2
\end{array}\right]=\alpha\quad \textnormal{and} \quad\left[\begin{array}{cc}
	g&0\\
	0&1
\end{array}\right]\left[\begin{array}{c}
\gamma_1\\
\gamma_2
\end{array}\right]=\left[\begin{array}{c}
\beta_1\\
\beta_2
\end{array}\right].$$
It is easy to see, that $\gamma=\left[\begin{array}{c}
	\alpha\\
	\beta_2
\end{array}\right]$ satisfies the above conditions. Therefore $Y$ is the pullback given by the biderivation $\delta\cdot g$.  

{\em Proof of part $(ii):$} Let $G$ be given by $\Xi:\F_q[t]\lra {\mathrm{Mat}}_{r}(K\{\tau\})$. It is easy to check that the following diagram:
$$\xymatrix @C+2pc @R+1.5pc{0\ar[r] & F\ar[r]^{\wlozenie} \ar[d]^{f} & X \ar[r]^{\rzut} \ar[d]^{\left[\begin{array}{cc}
			1&0\\
			0&f
		\end{array}\right]} & E \ar[r] \ar[d]^{=}& 0\\
	0\ar[r] & G\ar[r]_{\wlozenie} & Y \ar[r]_{\rzut} & E \ar[r] & 0}$$
is commutative with exact rows, where $Y$ is given by the following map: 
$$\left[\begin{array}{cc}
	\phi & 0\\
	f\cdot \delta  & \Xi
\end{array}\right]:\F_q[t]\lra {\mathrm{Mat}}_{e+d}(K\{\tau\}).$$
The proof of the  universal property for a pushout  is similar to that  for a pullback presented in part $(i).$ 
 \end{proof}
 
 \begin{rem}\label{uw:nowa} Notice that from  Theorem  \ref{thm:pull-back_push-out} it follows that the multiplication by 
 $a\in\F_q[t]$ of the short exact sequence  $\delta\in\Ext^1_\tau(\Phi,\Psi)$ is given by  the pullback of the sequence $\delta$ by the map $\Phi_a:\Phi\lra\Phi$ or equivalently by the  pushout of the sequence  $\delta$ by the map $\Psi_a:\Psi\lra \Psi$ \end{rem}

 Let $\Lambda$ be a ring,  $0\rightarrow A\rightarrow B\rightarrow C\rightarrow 0$ an exact sequence of $\Lambda$-modules and $D$  a $\Lambda$-module.
 It is a standard result in homological algebra (cf. \cite[Theorem 3.4]{Mac} or \cite[Theorem 5.2]{hs}) that one has a six term exact sequence 
 called $\Hom - \Ext$ sequence in the second variable:
 \begin{align*}
0\lra \Hom_{\Lambda}(D,A)\uplra{i\circ-}  \Hom_{\Lambda}(D,B)\uplra{\pi\circ-}  \Hom_{\Lambda}(D,C)\lra  \\
		\uplra{\delta\circ-} \Ext_{\Lambda}^1(D,A)\uplra{-i\circ-}  \Ext_{\Lambda}^1(D,B)\uplra{-\pi\circ-}  \Ext_{\Lambda}^1(D,C ).\numberthis \label{exseq}
		\end{align*}
Dually, one has the following  $\Hom - \Ext$ sequence in the first variable:	
\begin{align*}
		0\lra \Hom_{\Lambda}(C,D)\uplra{-\circ \pi}  \Hom_{\Lambda}(B,D)\uplra{-\circ i}  \Hom_{\Lambda}(A,D)\uplra{-\circ\delta}  \\
		\lra \Ext_{\Lambda}^1(C,D)\uplra {-\circ (-\pi)} \Ext_{\Lambda}^1(B,D)\uplra{-\circ (-i)}  \Ext_{\Lambda}^1(A,D). \numberthis \label{exseq1}
	\end{align*}	
		
		These sequences in general can be continued by higher $\Ext$ bifunctors. However, if  $\Lambda$ is a P.I.D. (or more generally a Dedekind ring) the last maps in (\ref{exseq}) and (\ref{exseq1}) are surjections \cite[Corollary 5.7]{hs}, \cite{i59}.

For exact sequences in the category of \tm modules we obtain analogous exact sequences of ${\mathbb F}_q[t]$-modules.
		
The following example shows that in general for \tm modules $(E,\Phi)$ and  $(F,\Psi)$ we have $\Hom_{\tau}(F,E)\subsetneq\Hom_{{\mathbb F}_q[t]}(G_{a}(K),G_a(K)).$ 
\begin{ex}\label{exm}
Let $\phi$ be a Drinfeld module  of rank $r$ defined over $K={\mathbb F}_q(t)$ i.e. $\phi_t={\sum}_{i=0}^ra_i\tau^i,\quad a_i\in{\mathbb F}_q[t].$ It is well-known that $\Hom_{\tau}(\phi, \phi)$ is a projective ${\mathbb F}_q[t]$-module of rank at most $r$ (cf. \cite{th04}). However, the Mordell-Weil group $\phi(K)=K$ of $\phi$ as an $\F_q[t]$-module is a direct sum 
of a finite torsion module and a free $\F_q[t]$-module on $\aleph_0$ generators (cf.\cite{p95}) This shows that 
$\Hom_{\tau}(\phi,\phi)\subsetneq\Hom_{{\mathbb F}_q[t]}(G_{a}(K),G_a(K)).$
\end{ex}
The following theorem gives an explicit description of the corresponding six term exact sequences:
		
\begin{thm}\label{thm:long_sequence}
Let 
$$\delta: \quad 0 \lra F\uplra{i} X\uplra{\pi} E\lra 0$$
 be a short exact sequence of \tm modules given by the biderivation $\delta$ and let $G$ be a \tm module. 
\begin{itemize}
	\item[$(i)$] There is an exact sequence of $\F_q[t]-$modules:
	\begin{align*}
		0\lra \Hom_{\tau}(G,F)\uplra{i\circ-}  \Hom_{\tau}(G,X)\uplra{\pi\circ-}  \Hom_{\tau}(G,E)\lra  \\
		\uplra{\delta\circ-} \Ext^1_{\tau}(G,F)\uplra{-i\circ-}  \Ext^1_{\tau}(G,X)\uplra{-\pi\circ-}  \Ext^1_{\tau}(G,E)\rightarrow 0.
	\end{align*}
	\item[$(ii)$]  There is an exact sequence of $\F_q[t]-$modules:
	\begin{align*}
		0\lra \Hom_{\tau}(E,G)\uplra{-\circ \pi}  \Hom_{\tau}(X,G)\uplra{-\circ i}  \Hom_{\tau}(F,G)\uplra{-\circ\delta}  \\
		\lra \Ext^1_{\tau}(E,G)\uplra {-\circ (-\pi)} \Ext^1_{\tau}(X,G)\uplra{-\circ (-i)}  \Ext^1_{\tau}(F,G)\rightarrow 0. 
	\end{align*}
\end{itemize} 
\end{thm}
\begin{proof}
	We will give a proof of  part $(i).$ The proof for  part $(ii)$ is similar and is left to the reader.
	
	Recall that $i=\wlozenie$, $\pi=\rzut$ and \tm modules $F$, $E$ and $X$  are given by the maps $\Psi:\F_q[t]\lra {\mathrm{Mat}}_{e}(K\{\tau\})$, $\Phi:\F_q[t]\lra {\mathrm{Mat}}_{d}(K\{\tau\})$ and
	$$\left[\begin{array}{cc}
		\Phi & 0\\
		\delta & \Psi
	\end{array}\right]:\F_q[t]\lra {\mathrm{Mat}}_{e+d}(K\{\tau\}).$$ Assume that $G$ is given by $\Xi:\F_q[t]\lra {\mathrm{Mat}}_{r}(K\{\tau\})$.
	The exactness at $\Hom_{\tau}(G,F)$, $\Hom_{\tau}(G,X)$ is obvious from the form of maps $i$ and $\pi$. 
	
Now we consider the exactness at $\Hom_{\tau}(G,E)$. Let $\left[\begin{array}{c}
		f_1\\
		f_2
	\end{array}\right]:G\lra X$ be a map of \tm modules. Hence there is an equality: 
	$$\left[\begin{array}{c}
		f_1\\
		f_2
	\end{array}\right]\Xi=
\left[\begin{array}{cc}
	\Phi & 0\\
	\delta & \Psi
\end{array}\right]
\left[\begin{array}{c}
	f_1\\
	f_2
\end{array}\right].$$
This implies 
\begin{equation}\label{star}
\quad f_1\Xi=\Phi f_1\quad\textnormal{and}\quad f_2\Xi=\delta f_1+\Psi f_2.
\end{equation}
We will prove, that an exact sequence given by the biderivation:
$$\delta\circ\pi\circ\left[\begin{array}{c}
	f_1\\
	f_2
\end{array}\right]=\delta\circ
\rzut \circ\left[\begin{array}{c}
f_1\\
f_2
\end{array}\right]=\delta\circ f_1=\delta f_1$$
splits. Consider the following  diagram: 
$$\xymatrix @C+2pc @R+1.5pc{
	&0\ar[r] & F\ar[r]^{\wlozenie} \ar[d]^{=} & G\oplus F\ar[r]^{\rzut} \ar[d]^{\left[\begin{array}{cc}
			1&0\\
			f_2&1
		\end{array}\right]} & G \ar[r] \ar[d]^{=}& 0\\
\delta f_1:	&0\ar[r] & F\ar[r]_{\wlozenie} & Y \ar[r]_{\rzut} & G \ar[r] & 0},$$
where the lower row is given by  the biderivation $\delta f_1$, i.e. $Y$ is defined by the map: 
$$\left[\begin{array}{cc}
	\Xi & 0\\
	\delta f_1 & \Psi
\end{array}\right]:\F_q[t]\lra {\mathrm{Mat}}_{e+r}(K\{\tau\}).$$
It is easy to see, that this diagram is commutative, with exact rows. From $(\ref{star})$ 
 the middle vertical map is a morphism of \tm modules that is also an isomorphism. 
 Therefore the sequence $\delta f_1$ splits.
 
On the other hand, assume that for some morphism  of \tm modules $f_1:G\lra E$ the sequence given by the biderivation
$\delta f_1$ splits. Therefore $\delta f_1\in\Derin(\Xi,\Psi)$, so there is $U\in {\mathrm{Mat}}_{e+r}(K\{\tau\})$  such that 
\begin{equation}\label{nabla}
\quad \delta f_1=\delta^{(U)}=U\Xi-\Psi U.
\end{equation}
Then $f=\left[\begin{array}{c}
	f_1\\
	U
\end{array}\right]:G\lra X$ is a morphism of \tm modules. Indeed 
\begin{align*}
	\left[\begin{array}{c}
		f_1\\
		U
	\end{array}\right]\Xi=
	\left[\begin{array}{c}
	f_1\Xi\\
	U\Xi
\end{array}\right]=
\left[\begin{array}{c}
	\Phi f_1\\
	\delta f_1+ \Psi U
\end{array}\right] = 
\left[\begin{array}{cc}
	\Phi & 0\\
	\delta  & \Psi
\end{array}\right]
\left[\begin{array}{c}
	f_1\\
	 U
\end{array}\right]
\end{align*}
where the second equality follows from ($\ref{nabla}$) and the fact that $f_1:G\lra E$ is a morphism of \tm modules. 
Hence $\pi\circ f= \rzut \left[\begin{array}{c}
	f_1\\
	U
\end{array}\right]=f_1$, which shows  exactness at  $\Hom_{\tau}(G,E)$. 

For the exactness at $\Ext^1_{\tau}(G,F)$ let $f:G\lra E$ be a morphism of \tm modules. We will prove that the extension given by the biderivation $-i\delta f$ splits. Recall that $\delta f$ determines  the following short exact sequence:
$$0\lra F\lra Y\lra G\lra 0,$$
where $Y$ is defined by the map: 
	$$\left[\begin{array}{cc}
	\Xi & 0\\
	\delta f & \Psi
\end{array}\right]:\F_q[t]\lra {\mathrm{Mat}}_{r+e}(K\{\tau\}).$$ Then the biderivation $-i\delta f$ 
determines  the following short exact sequence:
$$0\lra X\lra \widehat{Y}\lra G\lra 0,$$
where $\widehat{Y}$ is defined by the map: 
$$\left[\begin{array}{c|cc}
	\Xi & 0 &0\\ \hline
	0& \Phi & 0\\
	-\delta f & \delta & \Psi
\end{array}\right]:\F_q[t]\lra {\mathrm{Mat}}_{r+e}(K\{\tau\}).$$
It is easy to check, that the following diagram:
$$\xymatrix  @C+1pc @R+3pc{
	0\ar[r] & X\ar[r]^{\wlozenie} \ar[d]^{=} & Y \ar[rr]^{\rzut} \ar[d]^{\left[\begin{array}{ccc}
			1&0&0\\
			-f&1&0\\
			0&0&1
		\end{array}\right]} && X \ar[r] \ar[d]^{=}& 0\\
	0\ar[r] & X \ar[r]_{\wlozenie} & G\oplus X \ar[rr]_{\rzut} && G \ar[r] & 0},$$
is commutative and the map $\left[\begin{array}{ccc}
	1&0&0\\
	-f&1&0\\
	0&0&1
\end{array}\right]$ is an isomorphism  of \tm modules. Therefore the sequence $-i\delta f$ splits. 

 Now consider the exact sequence:
$$\eta:\quad 0\lra F\lra Y\lra G\lra 0 \in \Ext^1_{\tau}(G,F),$$
where $Y$ is given by the map    
$$\left[\begin{array}{cc}
	\Xi & 0\\
	\eta & \Psi
\end{array}\right]:\F_q[t]\lra {\mathrm{Mat}}_{r+e}(K\{\tau\}),$$
and assume that the sequence 
$$-i\circ\eta:\quad 0\lra X\lra \widehat Y\lra G\lra 0 \in \Ext^1_{\tau}(G,F)$$
splits, where 
$\widehat{Y}$ is defined by the map 
$$\left[\begin{array}{c|cc}
	\Xi & 0 &0\\ \hline
	0& \Phi & 0\\
	-\eta & \delta & \Psi
\end{array}\right]:\F_q[t]\lra {\mathrm{Mat}}_{r+e}(K\{\tau\}).$$
Because $-i\circ \eta$ splits, then the biderivation 
$-i\circ \eta\in\Derin\Bigg(\Xi, 
\left[\begin{array}{cc}
	\Phi & 0\\
	\delta & \Psi
\end{array}\right]\Bigg).$ 
Hence, there is $U=
\left[\begin{array}{c}
	u_1\\
	u_2 
\end{array}\right]
\in {\mathrm{Mat}}_{e+d\times r}(K\{\tau\})$ such that 
\begin{align*}
\left[\begin{array}{c}
		0\\
		-\eta 
	\end{array}\right]= \delta^{(U)}=
\left[\begin{array}{c}
	u_1\\
	u_2 
\end{array}\right] \Xi - 
\left[\begin{array}{cc}
	\Phi & 0\\
	\delta & \Psi
\end{array}\right]
\left[\begin{array}{c}
	u_1\\
	u_2 
\end{array}\right] = 
\left[\begin{array}{c}
	u_1\Xi-\Phi u_1\\
	u_2 \Xi - \delta u_1- \Psi u_2
\end{array}\right].
\end{align*}
Therefore $u_1$ is a morphism of \tm modules and 
\begin{align*}
	\eta &= \delta u_1 + \Psi u_2 - u_2\Xi = \delta u_1 -\podwzorem{\Big(u_2\Xi - \Psi u_2\Big)}{=\delta^{(u_2)}}
	= \delta u_1 -\delta^{(u_2)}.
\end{align*}
Thus the biderivations $\eta$ and $\delta u_1$ determine the same extension in $\Ext^1_{\tau}(G,F)$, 
 which shows the exactness at $\Ext^1_{\tau}(G,F)$.

 Now we consider exactness at  $\Ext^1_{\tau}(G,X)$. 
 For $\eta\in\Ext^1_{\tau}(G,F)$ there is an equality  
 $$-\pi\circ\Big( -i\circ \eta \Big) = \pi \circ i (\eta) =0.$$
 Hence the sequence given by the biderivation $-\pi\circ\Big( -i\circ \eta \Big)$ splits. 
 
 On the other hand  assume, that $\eta=
 \left[\begin{array}{c}
 	\eta_1\\
 	\eta_2
 \end{array}\right]$ is the biderivation determining the exact sequence from $\Ext^1_{\tau}(G,X)$, such that 
$-\pi \circ \eta$ gives a split sequence in $\Ext^1_{\tau}(G,E)$. Then 
$$-\pi \circ \eta=-\rzut\circ\left[\begin{array}{c}
	\eta_1\\
	\eta_2
\end{array}\right]=-\eta_1\in\Derin(\Xi, \Phi).$$ 
Thus there is $u\in {\mathrm{Mat}}_{d\times r}(K\{\tau\})$ such that 
$-\eta_1=\delta^{(u)}=u\Xi-\Phi u$.
We put
$U=\left[\begin{array}{c}
	u\\
	0
\end{array}\right]\in {\mathrm{Mat}}_{d+e\times r}(K\{\tau\}).$ Then 
the inner biderivation $\delta^{(U)}\in \Derin\Bigg(\Xi, \left[\begin{array}{cc}
	\Phi & 0\\
	\delta & \Psi
\end{array}\right] \Bigg)$
has the following form:
\begin{align*}
	\delta^{(U)}=
	\left[\begin{array}{c}
		u\\
		0
	\end{array}\right] \Xi - 
	\left[\begin{array}{cc}
		\Phi & 0\\
		\delta & \Psi
	\end{array}\right]
	\left[\begin{array}{c}
		u\\
		0 
	\end{array}\right] = 
	\left[\begin{array}{c}
		u\Xi-\Phi u\\
		- \delta u
	\end{array}\right]=
\left[\begin{array}{c}
	\delta^{(u)}\\
	- \delta u
\end{array}\right].
\end{align*}
Therefore 
\begin{align*}
	\eta+\delta^{(U)}&=\left[\begin{array}{c}
		-\delta^{(u)}\\
		\eta_2
	\end{array}\right]+
	\left[\begin{array}{c}
		\delta^{(u)}\\
		- \delta u
	\end{array}\right]=
	\left[\begin{array}{c}
		0\\
		\eta_2-\delta u
	\end{array}\right]=
	-\left[\begin{array}{c}
		0\\
		1
	\end{array}\right] \circ \Big( \delta u-\eta_2 \Big)\\
&= -i\circ \big( \delta u-\eta_2 \big).
\end{align*} 
Thus the sequence corresponding to $\eta$ is given by the biderivation $-i\circ \big( \delta u-\eta_2 \big)$. This proves the  exactness at $\Ext^1_{\tau}(G,X)$.  
Now we will prove that the map: $ \Ext^1_{\tau}(G,X)\uplra{-\pi\circ-}  \Ext^1_{\tau}(G,E)$ is a surjection.

Let
\begin{equation}
	\gamma: \quad 0 \lra E\uplra{\wlozenie} Y\uplra{\rzut} G\lra 0,
	\end{equation}
be an element of $\Ext^1_{\tau}(G,E)$ where $Y$ is given by the following map:
\begin{equation}\label{eeq}
\left[\begin{array}{cc}
	\Xi & 0\\
	\gamma & \Psi
\end{array}\right] : {\mathbb F}_q[t]\rightarrow {\mathrm{Mat}}_{e+r}(K\{\tau\}).
\end{equation}
Then there exists the following commutative diagram with exact rows:
\begin{equation}\label{dgg}
\xymatrix  @C+1pc @R+3pc{
0\ar[r] & X\ar[r]^{\left[\begin{array}{cc}0&0\\1&0\\0&1\end{array}\right]} \ar[d]^{-\rzut} & Z \ar[rr]^{\rzut} \ar[d]^{\left[\begin{array}{ccc}
			1&0&0\\
			0&-1&0
		\end{array}\right]} && G \ar[r] \ar[d]^{=}& 0\\
	0\ar[r] & E \ar[r]_{\wlozenie} & Y \ar[rr]_{\rzut} && G \ar[r] & 0},\end{equation}
where $Z$ is given by the following map:
$$\left[\begin{array}{ccc}
	\Xi&0&0\\
	-\gamma&\Phi&0\\
	0&\delta&\Psi
\end{array}\right]: {\mathbb F}_q[t] \rightarrow {\mathrm{Mat}}_{e+d+r}(K\{\tau\}).$$
Notice that
\begin{align*}
	\left[\begin{array}{ccc}
		1&0&0\\
		0&-1&0
	\end{array}\right] 
	\left[\begin{array}{ccc}
	\Xi&0& 0\\
	-\gamma&\Phi&0\\	
	0&\delta & \Psi
	\end{array}\right]=
	\left[\begin{array}{ccc}
		\Xi&0&0\\
		\gamma&-\Phi&0 
	\end{array}\right] = 
	\left[\begin{array}{cc}
		\Xi &0\\
		\gamma\Phi
	\end{array}\right]
\left[\begin{array}{ccc}
	1&0&0\\
	0&-1&0
\end{array}\right].
\end{align*}
Thus  (\ref{dgg}) is the diagram of morphisms of \tm modules and $\left(\begin{array}{c}-\gamma\\0\end{array}\right)\rightarrow \gamma$, where $\left(\begin{array}{c} -\gamma\\0\end{array}\right)$ is the extension given by the upper row 
of (\ref{dgg}).
\end{proof}

\begin{rem}
	In the Theorem \ref{thm:long_sequence} we assumed that the short exact sequence is given by the biderivation $\delta$ and in the result we obtained simple formula for the morphism from $\Hom_\tau$ to $\Ext^1_\tau$. In the case where the  short exact sequence is not given by the biderivation, but it  is isomorphic to a sequence given by the biderivation, the six-term exact sequences exist. The aforementioned  isomorphism of short exact sequences induces the isomorphism of the corresponding six-term sequences.
\end{rem}

We finish this section with the application of the six-term exact sequence.

\begin{ex}
Let $F$ (resp. $E$) be a Drinfeld module given by $\phi_t=\theta+\tau^3$ (resp. $\psi_t=\theta+\tau^2$) and consider 
the exact sequence of \tm modules 
\begin{equation}\label{er}
0\rightarrow F\rightarrow X\rightarrow E\rightarrow 0,
\end{equation}
 where $X$ is the extension given by the biderivation $\delta_t=1+\tau$ i.e. $X$ is the \tm module given by $\Gamma_t=\left[\begin{array}{cc}\theta+\tau^2&0\\1+\tau&\theta+\tau^3\end{array}\right]$.
Let $C$ be the Drinfeld module given  by ${\eta}_t=\theta+\tau.$ 
Since $\rk F>\rk C$ we have $\Hom_{\tau} (F,C)=0$ and from the six-term exact  sequence we obtain the short exact sequence:
\begin{equation}\label{exs}
0\rightarrow \Ext^1_{\tau}(E,C)\rightarrow \Ext^1_{\tau}(X,C)\rightarrow \Ext^1_{\tau}(F,C)\rightarrow 0
\end{equation}
One readily verifies that $\Ext^1_{\tau}(E,C)$ (resp. $\Ext^1_{\tau}(F,C)$) is the \tm module given by $\Psi_t=\left[\begin{array}{cc}\theta&0\\\tau&\theta+\tau^2\end{array}\right]$ (resp. $\Phi_t=\left[\begin{array}{ccc}\theta&0&0\\\tau&\theta &\tau^2\\
0&\tau&\theta
\end{array}\right]$).

We have:
\begin{align*}\label{xtt}
\numberthis \Ext^1_{\tau}(X,C)={\mathrm{Der(X,C)}}/{\mathrm{Der_{in}(X,C)}}\\=K\{ \tau \}^2/ \Big\langle\delta^{\left[\begin{array}{cc}c{\tau}^k,&0\end{array}\right]}, \delta^{\left[\begin{array}{cc}0,&c{\tau}^k\end{array}\right]}\mid c\in K, k\in \mathbb{Z}_{\geq 0}\Big\rangle,
\end{align*}
where 
\begin{align*}
\delta^{\left[\begin{array}{cc}c{\tau}^k,&0\end{array}\right]}&=
\left[\begin{array}{cc}c{\tau}^k,&0\end{array}\right]\left[\begin{array}{cc}\theta+\tau^2&0\\
1+\tau&\theta+\tau^3
\end{array}\right]-(\theta+\tau)\left[\begin{array}{cc}c{\tau}^k,&0\end{array}\right]\\
&=
\left[\begin{array}{cc}\delta^{(c{\tau}^k)},&0\end{array}\right]\quad\textnormal{where}\quad  \delta^{(c{\tau}^k)} \in{\mathrm{Der}_{in}(E,C)}
\end{align*}
\begin{align*}
	\delta^{\left[\begin{array}{cc}0,&c{\tau}^k\end{array}\right]}&=
	\left[\begin{array}{cc}0,&c{\tau}^k\end{array}\right]\left[\begin{array}{cc}\theta+\tau^2&0\\
		1+\tau&\theta+\tau^3
	\end{array}\right]-(\theta+\tau)\left[\begin{array}{cc}0,&c{\tau}^k\end{array}\right]\\
	&=
	\left[\begin{array}{cc}c\tau^k+c\tau^{k+1},&\delta^{(c{\tau}^k)}\end{array}\right]\quad\textnormal{where}\quad  \delta^{(c{\tau}^k)} \in{\mathrm{Der}_{in}(F,C)}
\end{align*}

Thus first reducing the second coordinates by the elements $\delta^{\left[\begin{array}{cc}0,&c{\tau}^k\end{array}\right]}$ and then 
the first coordinates by $\delta^{\left[\begin{array}{cc}c{\tau}^k,&0\end{array}\right]}$ one can see that:
\begin{equation}\label{exxx1}
\Ext^1_{\tau}(X,C)\cong\{\left[\begin{array}{cc}c_0+c_1\tau,&d_0+d_1\tau+d_2\tau^2\end{array}\right] \,\mid \,c_i, d_i\in K\}
\end{equation}	
Enumerating the basis elements of (\ref{exxx1}) lexicographically i.e. $$\left[\begin{array}{cc}0,&1\end{array}\right], 
\left[\begin{array}{cc}0,&\tau\end{array}\right], \left[\begin{array}{cc}0,&{\tau}^2\end{array}\right],  \left[\begin{array}{cc}1,&0\end{array}\right], \left[\begin{array}{cc}1,&\tau\end{array}\right]$$ and computing 
$t*\left[\begin{array}{cc}0,&d_i\cdot \tau^i\end{array}\right],\,i=0,1,2$ and
$t*\left[\begin{array}{cc}c_i\cdot\tau^i,&0\end{array}\right],\, i=0,1,$ in a similar to that in Section \ref{examples} way, one obtains that $\Ext^1_{\tau}(X,C)$ is a \tm module
defined by the matrix:
\begin{equation}
\Omega_t=\left[\begin{array}{ccccc}\theta&0&0&0&0\\
\tau&\theta&\tau^2&0&0\\
0&\tau&\theta&0&0\\
0&0&-\tau&\theta&0\\
0&0&-\tau&\tau&\theta+\tau^2
\end{array}\right]
\end{equation}
Therefore
\begin{equation}\label{er1}
\left[\begin{array}{cc}\Phi_t & 0\\
\Delta_t & \Psi_t \end{array}\right],
\end{equation}
where $\Delta_t=\left[\begin{array}{ccc}0 & 0 & -\tau\\
0 & 0 & -\tau\end{array}\right].$
Thus we showed that the exact sequence (\ref{exs}) is an extension of \tm modules.

Change of $\delta_t$ in (\ref{er}) to $\delta_t =1+\tau^3$ results in the change of $\Delta_t$ in (\ref{er1}). 
For the new data 
$\Delta_t=\left[\begin{array}{ccc}0 & 0 & -\tau\\
0 & 0 & (\theta - \theta^{(1)})\tau+\tau^4\end{array}\right].$ 
\end{ex}

Similar arguments to that in the proof of  Proposition \ref{prop:ext_as_t_module}, show that the following assertions hold true.
\begin{prop}\label{prop:sequence_exts}
	Let $0\lra F\lra X\lra E\lra 0$ be an exact sequence of \tm modules, where $F$ and $E$ are Drinfeld modules, and let $G$ be a a Drinfeld module. 
	\begin{itemize}
		\item[$(i)$] If $\rk G<\rk F$ and $\rk G<\rk E$, then there is a short exact sequence of \tm modules
		$$0\lra \Ext^1_\tau(E,G)\lra \Ext^1_\tau(X,G) \lra \Ext^1_\tau(F,G) \lra 0.$$
			\item[$(i^D)$] If $\rk G>\rk F$ and $\rk G>\rk E$, then there is a short exact sequence of \tm modules
		$$0\lra \Ext^1_\tau(G,F)\lra \Ext^1_\tau(G,X) \lra \Ext^1_\tau(G,E) \lra 0.$$
	\end{itemize}
\end{prop} 
As an immediate consequence of the Theorem \ref{thm:short_sequence_t_modules} we obtain the following theorem:
\begin{thm}\label{thm:last}
	Let $0\lra F\lra X\lra E\lra 0$ be an exact sequence of \tm modules, where $F$ and $E$ are Drinfeld modules, and let $G$ be a  Drinfeld module. 
	\begin{itemize}
		\item[$(i)$] If $\rk G<\rk F$ and $\rk G<\rk E$, then there is a short exact sequence of \tm modules
		$$0\lra \Ext_{0,\tau}(X,G)\lra \Ext^1_\tau(X,G) \lra \G_a^2 \lra 0.$$
		\item[$(i^D)$] If $\rk G>\rk F$ and $\rk G>\rk E$, then there is a short exact sequence of \tm modules
		$$0\lra \Ext_{0,\tau}(G,X)\lra \Ext^1_\tau(G,X) \lra G_a^2 \lra 0.$$
	\end{itemize}
\end{thm}
\begin{rem}
	In the case where $K$ is perfect, one  can prove the corresponding ''\tsm versions'' of the Proposition \ref{prop:sequence_exts} and Theorem \ref{thm:last}. 
\end{rem}

\section{Extensions of  dual \tm motives}\label{dmot}
Now,  recall the notion of  a dual  \tm motive (cf. \cite{bp20}). 
\begin{dfn}
	Let $K$ be a perfect field and let  $K[t,\sigma]$ be the polynomial ring satisfying the following relations:
	\begin{equation}\label{skew}
		tc=ct, \quad t\sigma=\sigma t, \quad \sigma c=c^{(-1)}\sigma, \quad c\in K.
	\end{equation}
	A dual \tm motive is a left $K[t,\sigma]$-module that is free and finitely generated over $K\{\sigma\}$ and for which  there exists an
	$l\in {\mathbb N}$ such that $(t-\theta)^l(H/\sigma H)=0.$ A morphism of dual \tm motives is a morphism of $K[t,\sigma]$-modules. 
\end{dfn}

For a \tm module $\Phi \rightarrow {\mathrm{Mat}}_d(K\{\tau\})$ let $H(\Phi)={\mathrm{Mat}}_{1\times d}(K\{\sigma\})$ 
i.e. a free $K\{\tau\}$-module on $d$ generators. Equip $H(\Phi)$ with the following $\F_q[t]$-action:
\begin{equation}\label{action}
	a\cdot h=h{\Phi}_a^{\sigma}, \quad\textnormal{for}\quad  h\in H(\Phi), \quad a\in {\mathbb F}_q[t].
\end{equation}

Every morphism of \tm modules $f:\Phi\lra \Psi$ induces a morphism of  dual \tm motives
$H(f):H(\Phi)\lra H(\Psi)$ defined by the following formula:
\begin{equation}
	H(f)(h)=h\cdot f^{\sigma} \quad\textnormal{for}\quad  h\in H(\Phi).
\end{equation}
Vice versa every morphism of  dual \tm motives  $g: H(\Phi)\rightarrow H(\Psi)$ comes from a morphism of  \tm modules.

From $H(\Phi)$ one can recover $\Phi$ as:
\begin{equation}\label{hphi}
	\frac{H(\Phi)}{(\sigma -1)H(\Phi)}\cong (\Phi , K^d)
\end{equation}
The following theorem was proved by G. Anderson.
\begin{thm}[Anderson]\label{And} The correspondence between dual \tm motives and \tm modules
	over a perfect field $K$ gives an equivalence of categories.
\end{thm}

\begin{dfn}
We call a sequence of dual \tm motives:
\begin{equation}\label{exct}
0\lra H(\Psi)\lra H(\Xi)\lra H(\Phi)\lra 0,
\end{equation}
exact if it is exact as a sequence of  $\F_q[t]-$modules. 
The space of  all exact sequences  (\ref{exct}) for fixed $\Psi$ and $\Phi$ will be denoted as  $\Ext^1_{{\cal M}_t^{\vee}}(H(\Phi), H(\Psi))$. 
\end{dfn}

Similarly as in the case of  \tm modules the space $\Ext^1_{{\cal M}_t^{\vee}}$ can be endowed with the structure of an $\F_q[t]-$module, where the multiplication by an  element $a\in\F_q[t]$ is given by the pushout of the map $H(\Psi_a):H(\Psi)\lra H(\Psi)$, (cf. Remark  \ref{uw:nowa}.) 
 
In general an equivalence of categories need not preserve exact sequences, so an isomorphism of corresponding
spaces of extensions is not an obvious fact.
However, we have the following:
\begin{thm}
	Let $\Phi$ and $\Psi$ be \tm modules. Then there exists an isomorphism of ${\mathbb F}_q[t]$-modules:
	\begin{equation}\label{diso}
		\Ext^1_{\tau}(\Phi,\Psi)\cong \Ext^1_{{\cal M}_t^{\vee}}(H(\Phi), H(\Psi)).
	\end{equation}
\end{thm}
\begin{proof}
	Let $\Phi$  be a \tm module and let $h=[w_1(\sigma),\dots, w_d(\sigma)]\in H(\Phi)$, where 
	$w_i(\sigma)=\sum_{j=0}^{n_i}a_{i,j}\sigma^i$. Then 
	\begin{align*}
		h&= \big[w_1(\sigma),\dots, w_d(\sigma)\big]=\sum_{j=0}^{\max\{n_i\}}\big[a_{1,j},\dots, a_{d,j}\big]\sigma^j\in \bigoplus_{j=0}^\infty K^d\sigma^j.
	\end{align*}
 Thus every dual  \tm motive $H(\Phi)$, as an $\F_q[t]-$module,  can be viewed as an element of the space  ${\bigoplus}_{i=0}^{\infty}(K^d)_i$, where the action of $a\in\F_q[t]$ on the  $i-$th component is given by the following formula:
	\begin{equation}\label{ac1}
		a\cdot k=k\sigma^{i}\Phi_a^{\sigma},\quad \textnormal{for} \quad k\in (K^d)_i\quad \textnormal{and} \quad a\in A
	\end{equation}
Notice that every component  $(K^d)_i$ with the action  (\ref{ac1}) is an $F_q[t]-$module.
 
So we see that starting with the exact sequence of \tm modules
	$$0\rightarrow (\Psi, K^e)\rightarrow (\Xi, K^{d+e})\rightarrow (\Phi, K^d)\rightarrow 0$$
	one obtains exact sequences of $\F_q[t]-$modules
	$$0\lra (K^d)_i\lra (K^{d+e})_i\lra (K^{e})_i\lra 0\quad \textnormal{for all}\quad i=0,1,2,\dots$$
	Exactness follows from the formula (\ref{ac1}).
	Thus we get an exact sequence of  $\F_q[t]-$modules
	\begin{equation}
		0\rightarrow {\bigoplus}_{i=0}^{\infty}(K^d)_i\rightarrow {\bigoplus}_{i=0}^{\infty}(K^{d+e})_i\rightarrow  {\bigoplus}_{i=0}^{\infty}(K^{e})_i\rightarrow 0
	\end{equation}
which in turn  yields an  exact sequence:
$$0\lra H(\Psi)\lra H(\Xi)\lra H(\Phi)\lra 0.$$
So, $H(-)$ preserves exact sequences.
	
	Now assume that we have an exact sequence of dual \tm motives
	$$0\rightarrow H(F)\rightarrow H(\Xi)\rightarrow H(E)\rightarrow 0.$$
	One easily verifies that the induced sequence:
	$$0\rightarrow \frac{H(F)}{(\sigma-1)H(F)}\rightarrow \frac{H(\Xi)}{(\sigma-1)H(\Xi)}\rightarrow \frac{H(E)}{(\sigma-1)H(E)}\rightarrow 0$$
	is an exact sequence of $\F_q[t]$-modules.
	
From the definition of   $\F_q[t]-$module structure  on $\Ext^1_{{\cal M}_t^{\vee}},$
taking into account Theorem  \ref{thm:pull-back_push-out} and Remark \ref{uw:nowa}, it easily follows that  $H(-)$ induces an  $\F_q[t]-$module isomorphism: 	$$\Ext^1_{\tau}(\Phi,\Psi)\cong \Ext_{{\cal M}_t^{\vee}}^1(H(\Phi), H(\Psi)).$$ 
\end{proof}

We also have the following theorem for dual $t$-motives:

\begin{thm}\label{thm:long_sequencetmot}
Let $0\rightarrow M_1 \rightarrow M\rightarrow M_2\rightarrow 0$ be an exact sequence of dual \tm motives and let 
$N$ be a dual \tm motive.
 \begin{itemize}
	\item[$(i)$] There is an exact sequence of $\F_q[t]-$modules:
	\begin{align*}
	0\lra \Hom_{{\cal M}_t^{\vee}}(N,M_1)\lra  \Hom_{{\cal M}_t^{\vee}}(N,M)\lra  \Hom_{{\cal M}_t^{\vee}}(N,M_2)\lra  \\
		\lra \Ext^1_{{\cal M}_t^{\vee}}(N,M_1)\lra  \Ext^1_{{\cal M}_t^{\vee}}(N,M)\lra  \Ext^1_{{\cal M}_t^{\vee}}(N,M_2)\rightarrow 0. 	
	\end{align*}
	\item[$(ii)$]  There is an exact sequence of $\F_q[t]-$modules:
	\begin{align*}
		0\lra \Hom_{{\cal M}_t^{\vee}}(M_2,N)\lra  \Hom_{{\cal M}_t^{\vee}}(M,N)\lra  \Hom_{{\cal M}_t^{\vee}}(M_1,N)\lra  \\
		\lra \Ext^1_{{\cal M}_t^{\vee}}(M_2,N)\lra  \Ext^1_{{\cal M}_t^{\vee}}(M,N)\lra  \Ext^1_{{\cal M}_t^{\vee}}(M_1,N)\rightarrow 0. 
	\end{align*}
\end{itemize} 
\end{thm}
\begin{proof}
Pick \tm modules $\Phi_1, \Phi , \Phi_2$  and $\Psi$ such that 
\begin{equation}\label{izzo}
M_1\cong H(\Phi_1),\quad M\cong H(\Phi) \quad M_2\cong H(\Phi_2). \quad {\mathrm{and }} \quad N \cong H(\Psi).
\end{equation}
This is possible by Theorem \ref{And}.
Now follow the proof of the Theorem \ref{thm:long_sequence} for the \tm modules $\Phi_1, \Phi , \Phi_2$  and $\Psi$ and finally apply the functor $H$ again.
\end{proof}

\section*{Acknowledgement}
The authors would like to thank the anonymous referee for numerous corrections,  valuable questions,	comments and suggestions. In fact,   Sections \ref{Phipsi}, \ref{Carlitz}  and  \ref{dmot} were added as a result of these questions.

{}

\end{document}